\documentclass[12pt]{amsart}
\usepackage{amssymb}
\usepackage{verbatim}
\usepackage[toc,page]{appendix}
\usepackage{mathrsfs}
\usepackage{mathtools}

\usepackage{xcolor}

\newtheorem{thm}{Theorem}[section]

\newtheorem{lem}[thm]{Lemma}
\newtheorem{rem}[thm]{Remark}

\theoremstyle{definition}

\numberwithin{equation}{section}
\newcommand{\Mod}[1]{\ (\textup{mod}\ #1)}

\newcommand{\Zag}{\mathscr{L}}% Zagier L-series
\newcommand{\MT}{\mathbf{MT}} % main term
\newcommand{\SZE}{\Sigma_{1}}%sum over n<2l
\newcommand{\SIN}{\Sigma_{2}}%sum over n>2l
\newcommand{\MTn}{\Sigma_{0}}%main term n=2l
\newcommand{\Resid}{\mathcal{R}}

\newcommand{\ups}{\upsilon}% coef of Zagier/zeta
\newcommand{\G}{\mathcal{G}}
\newcommand{\Hf}{\mathcal{H}}
\newcommand{\Cc}{\mathcal{C}}% coef in main term
\newcommand{\MAP}{\mathcal{M}}% arith part in main term
\newcommand{\ES}{\mathcal{E}}%sum over e of SG
\newcommand{\TM}{\vartheta_k}

\newcommand{\M}{\mathcal{M}}

\providecommand{\sgn}{\operatorname{sgn}}

\providecommand{\sym}{\operatorname{sym}}
\providecommand{\arcosh}{\operatorname{arcosh}}

\DeclareMathOperator{\res}{res}

\newcommand{\GenHyG}[5]{ {}_{#1}F_{#2} \left( \begin{matrix} #3 \\ #4 \end{matrix} ; #5 \right) }

\mathtoolsset{showonlyrefs}
\begin{document}

\title[Moments and Non-Vanishing of Maass Form Symmetric Square $L$-Functions]{Moments and Non-Vanishing of Maass Form Symmetric Square $L$-Functions in Short Intervals}

\begin{abstract}
Recently, Li obtained a mean Lindel\"of estimate for the cubic moment of the central values of Maass form symmetric square $L$-functions over short intervals $(T-H, T+H)$ of length $H \ge T^{18/19+\epsilon}$. We improve this result by showing that the estimate holds for $H \gg T^{6/7+\epsilon}$. The key ingredient in our proof is a new asymptotic formula for the second twisted moment of Maass symmetric square $L$-functions. Based on this formula, we also improve the lower bound for the proportion of non-vanishing central $L$-values in short intervals. Previously, even under the assumption of the Lindel\"of hypothesis for Dirichlet $L$-functions, the proportion of non-vanishing values in intervals of length $H = T^{\beta}$ was only known to be at least $\frac{3\beta-1}{4}$. We establish an unconditional lower bound of $\frac{7\beta-2}{8}$.

\end{abstract}

\author{Olga  Balkanova}
\address{Steklov Mathematical Institute of Russian Academy of Sciences, 8 Gubkina st., Moscow, 119991, Russia}
\email{balkanova@mi-ras.ru}

\author{Dmitry Frolenkov}
\address{HSE University and Steklov Mathematical Institute of Russian Academy of Sciences, 8 Gubkina st., Moscow, 119991, Russia}
%\thanks{The work was supported by the Theoretical Physics and Mathematics Advancement Foundation BASIS}
\email{frolenkov@mi-ras.ru}


\maketitle

%\tableofcontents

%%%%%%%%%%%%%%%%%%%%%%%%%%%%%%%%%%%%%%%%%%%%%%%%%%%%%%%%%%%%%%%%%%%%%%%%%%%%%%%%%%%%%%%%%%%%%%%%%%%%%%%%%%%%%%%%%%%%%%%%%%%%%%%%%%
%%%%%%%%%%%%%%%%%%%%%%%%%%%%%%%%%%%%%%%%%%%%%%%%%%%%%%%%%%%%%%%%%%%%%%%%%%%%%%%%%%%%%%%%%%%%%%%%%%%%%%%%%%%%%%%%%%%%%%%%%%%%%%%%%%

\section{Introduction}
The study of moments of symmetric square $L$-functions associated with automorphic forms has a rich history spanning the past thirty years, originating with the seminal papers of Luo--Sarnak \cite{LuoSarnakPG} and Iwaniec--Michel \cite{IwMi2001}. In these works, they obtained an upper bound of the correct order of magnitude for the second moment of symmetric square $L$-functions for Maass forms in the spectral aspect and holomorphic cusp forms in the level aspect, respectively. Although the result of Iwaniec--Michel remains unimproved, progress in other directions -- specifically, the weight aspect with an additional average over the weight, and the spectral aspect -- has been more promising. This discrepancy arises because the ratio $\log(\text{conductor}) : \log(\text{size of the family})$ equals $1$ in these cases, whereas it equals $2$ in the level aspect considered by Iwaniec and Michel. The averaged weight and spectral aspects share significant commonalities; heuristically, a breakthrough in one setting often translates to a similar result in the other. In the weight aspect, Khan \cite[Theorem 3.1]{Khan2010} obtained an asymptotic formula for the averaged twisted second moment
\begin{equation}\label{twisted 2mom WA}
\sum_{k}h\left(\frac{k}{K}\right)\sum_{f \in H_{2k}}^{h}\lambda_f(r^2)L^2(\sym^2f,1/2)=\frac{K}{\sqrt{r}}P_3(\log K)+O(r^{1/2}K^{\epsilon}), \quad r\ll K^{1-\epsilon}.
\end{equation}
The condition $r \ll K^{1-\epsilon}$ enabled Khan to obtain a non-vanishing proportion of $70.37\%$, but this range is too narrow to yield an asymptotic formula for the cubic moment. The case of the third moment was subsequently considered by Das and Khan in \cite[Theorem 1.2]{DasKhan}, where they proved the asymptotic formula
\begin{equation}\label{3mom WA}
\sum_{k}h\left(\frac{k}{K}\right)\sum_{f \in H_{2k}}^{h}L^3(\sym^2f,1/2)=KP_6(\log K)+O(K^{1/2+\epsilon}).
\end{equation}
Recently, the authors \cite[Theorem 1.2]{BF2momWA2026} improved Khan's second-moment result by establishing the following asymptotic formula for $r \ll K^{5/4-\epsilon}$:
\begin{equation}\label{twisted 2mom WA new}
\sum_{k}h\left(\frac{k}{K}\right)\sum_{f \in H_{2k}}^{h}\lambda_f(r^2)L^2(\sym^2f,1/2)=\frac{K}{\sqrt{r}}P_3(\log K)+O\left(\frac{r^{3/2}}{K^{3/2-\epsilon}}+\frac{(rK)^{\epsilon}}{\sqrt{r}}\right).
\end{equation}
This enables one to derive \eqref{3mom WA} directly from the twisted second moment. Heuristically, one can represent a single factor of $L(\sym^2 f, 1/2)$ via the approximate functional equation; the resulting contribution of the error term is then
\begin{equation}\label{3mom WA2}
\sum_{r\ll K^{1+\epsilon}}\frac{1}{\sqrt{r}}\left(\frac{r^{3/2}}{K^{3/2-\epsilon}}+\frac{(rK)^{\epsilon}}{\sqrt{r}}\right)\ll K^{1/2+\epsilon}.
\end{equation}
It is natural to expect that an asymptotic formula similar to \eqref{3mom WA} also holds for Maass forms, which constitutes one of the main results of the present paper. Nevertheless, in the context of Maass forms, it is more common to investigate moments over short intervals. In this direction, our goal is to determine the smallest window size $G$ for which the following mean Lindel\"of estimate holds:
\begin{equation}\label{Maass n-mom}
\sum_{|t_j-T|\ll G}\alpha_{j}L(\sym^2 u_{j},1/2)^n\ll TG^{1+\epsilon}.
\end{equation}
Here, $u_j$ is a Hecke--Maass cusp form for the full modular group with Laplace eigenvalue $\kappa_{j}=1/4+t_{j}^2$. In the case of $n=2$, the current record $G\gg T^{1/5+\epsilon}$ is due to Khan and Young \cite[Theorem 1.3]{KhYoung} (see also \cite{Frolsym2} for an alternative proof). In the case of $n=3$, Li \cite[Theorem 1.1]{Li2025} established \eqref{Maass n-mom} for $G \gg T^{18/19+\epsilon}$.

 For $G \ll T^{1-\epsilon}$, let
\begin{equation}\label{Mn def}
\M_n(r,1/2):=\sum_{j}e^{-(t_j-T)^2/G^2}\alpha_{j}\lambda_{j}(r^2)L(\sym^2 u_{j},1/2)^n.
\end{equation}
The main result of the paper is the following theorem.
%%%%%%%%%%%%%%%%%%%%%%%%%%%%%%%%%%%%%%%%%%%%%%%%%%%%%%55
\begin{thm}\label{main thm 3mom}
We have the asymptotic formula
\begin{equation}\label{3momAF0}
\M_3(1,1/2)
=\int_0^{\infty}te^{-(t-T)^2/G^2}P_6(\log t)dt+
O\left(\frac{T^{4+\epsilon}}{G^{5/2}}\right),
\end{equation}
which implies that for $G \gg T^{6/7+\epsilon}$,
\begin{equation}\label{mean Lindelof}
\sum_{T<t_j\le T+G}|L(\sym^2 u_{j},1/2)|^3\ll T^{1+\epsilon}G.
\end{equation}
\end{thm}
%%%%%%%%%%%%%%%%%%%%%%%%%%%%%%%%%%%%%%%%%%%%%%%%%%%%%%%%%%%%%%5
This result is an immediate consequence of our new asymptotic formula for the twisted second moment, which serves as the Maass form analogue of \eqref{twisted 2mom WA new} (cf. \cite[Theorem 1.2]{BF2momWA2026}).

%%%%%%%%%%%%%%%%%%%%%%%%%%%%%%%%%%%%%%%%%%%%%%%
\begin{thm}\label{thm:2mom average}
We have the asymptotic formula
\begin{multline}\label{2momAF0}
\M_2(r,1/2)
=\frac{2}{\pi^2}\int_0^{\infty}te^{-(t-T)^2/G^2}\sum_{e|r^2}
\frac{1}{\sqrt{re_1}}\Biggl(
\frac{1}{2}\log\frac{te_2}{r}\log^2\frac{t}{e_1e_2}-\frac{1}{6}\log^3\frac{t}{e_1e_2}+\\+
\log\frac{te_2}{r}P_1(\log\frac{t}{e_1e_2})+P_2(\log\frac{t}{e_1e_2})\Biggr)dt+
O\left(T^{\epsilon}\max(H,T^{2/3})+
\frac{T^{2+\epsilon}r^{3/2}}{G^{5/2}}
\right),
\end{multline}
where $P_j(x)$ denotes a polynomial of degree $j$, and $e=e_1e_2^2$ with $e_1$ square-free.

\end{thm}
%%%%%%%%%%%%%%%%%%%%%%%%%%%%%%%%%%%%%%%%%%%%%%%%%%%%%%%%%5
Another classical application of the asymptotic formula for the twisted second moment is the non-vanishing problem for $L(\sym^2 u_{j},1/2)$. Since Theorem \ref{thm:2mom average} improves upon \cite[Theorem 1.2]{BF2mom}, we are able to increase the non-vanishing proportion established in \cite[Theorem 1.1]{BF2mom}. This result is detailed in the following theorem.

%%%%%%%%%%%%%%%%%%%%%%%%%%%%%%%%%%%
\begin{thm}\label{thm:nonvan}
For any $\epsilon>0$ and sufficiently large $T$ we have
\begin{equation}\label{nonvan 1}
\sum_{\substack{T\leq t_j\leq T+H\\ L(\sym^2 u_{j},1/2)\neq 0}}\alpha_j\geq \left( 1-\frac{1}{(1+\Delta)^3}-\epsilon\right)
\sum_{\substack{\\T\leq t_j\leq T+H}}\alpha_j,
\end{equation}
where $H=T^{\beta}$ with $2/7<\beta<1$, and $\Delta<\frac{7\beta-2}{8}$.
\end{thm}
%%%%%%%%%%%%%%%%%%%%%%%%%%%%%%%%%%%
Previously, it was shown in \cite[Theorem 1.1]{BF2mom} that the inequality \eqref{nonvan 1} holds for (see \cite[(1.11)]{BF2mom})
\begin{equation}\label{Delta est1/6}
\Delta < \max\left(\beta - \frac{1}{2}, \frac{11\beta - 5}{16}\right),
\end{equation}
and conditionally on the Lindel\"of hypothesis for Dirichlet $L$-functions for
\begin{equation}\label{Delta est0}
\Delta < \frac{3\beta - 1}{4}.
\end{equation}
It is straightforward to see that our new unconditional estimate, $\Delta < \frac{7\beta - 2}{8}$, improves even the conditional bound \eqref{Delta est0}.

The proof of the key Theorem \ref{thm:2mom average} follows the general strategy used to establish \eqref{twisted 2mom WA new} in \cite{BF2momWA2026}. Applying the approximate functional equation to one of the two $L$-functions, we relate the second moment to the first moment. We then apply a reciprocity-type formula to the first moment, expressing it as a sum of the main term and sums of Zagier's $L$-series. This yields a formula for the second moment that consists of a diagonal main term and two double sums of Zagier's $L$-series. To handle these double sums, we transform them into a form amenable to the Voronoi summation formula for Zagier's $L$-series, which extracts the off-diagonal main terms and the error terms. To bound the remaining errors, we combine various techniques from \cite{BF2momWA2026} and \cite{Frolsym2}. 

For large twists, it turns out that evaluating the diagonal and off-diagonal parts of the main term individually is unfeasible. Remarkably, both components can be transformed into a unified structure that enables the evaluation of their sum. Specifically, their sum can be written as a contour integral of an odd function, which reduces to a specific residue at $z = 0$. It is worth mentioning that such a cancellation mechanism is not unprecedented; similar phenomena have previously appeared in the work of Kowalski--Michel--VanderKam \cite{KMV} and Soundararajan \cite{Sound}.

The paper is organized as follows. In Section \ref{sec:Notations}, we collect some standard facts regarding $L(\sym^2 u_{j},s)$, and in Section \ref{sec:1st mom}, we present a reciprocity-type formula for the first twisted moment. Section \ref{sec:2mom to 1mom} relates the second moment to the first moment, yielding an expression for the second moment as a sum of the diagonal main term and two off-diagonal sums. In Section \ref{sec: main term diagonal}, we begin the evaluation of the diagonal component of the main term. Sections \ref{sec:SIN} and \ref{sec:SZE} are devoted to the analysis of the two off-diagonal sums, based on the Voronoi summation formula for Zagier's $L$-series stated in Section \ref{sec:Voronoi}. In these sections, we establish the new estimate for the error term in the asymptotic formula for the twisted second moment and determine the off-diagonal part of the main term. In Section \ref{sec:Thm 2mom}, we complete the proof of Theorem \ref{thm:2mom average} by evaluating the combined sum of the diagonal and off-diagonal main terms, and consequently deduce Theorem \ref{thm:nonvan}. Finally, Section \ref{sec:3mom} establishes the new asymptotic formula for the cubic moment and verifies that the main term involves precisely a polynomial of degree six.

%%%%%%%%%%%%%%%%%%%%%%%%%%%%%%%%%%%%%%%%%%%%%%%%%%%%%%%%%%%%%%%%%%%%%%%%%%%%%%%%%%%%%%%%%%%
\section{Notations and preliminaries}\label{sec:Notations}
We denote by  $\{u_j\}$ an orthonormal basis of the space of Hecke-Maass cusp forms of level one. Let $\kappa_{j}=1/4+t_{j}^2$ (with  $t_j>0$)  be the eigenvalue of the hyperbolic Laplacian acting on $u_{j}$, and $\{\lambda_{j}(n)\}$ be the eigenvalues of Hecke operators acting on $u_{j}$.  Every $u_j$ possesses the following Fourier expansion
\begin{equation*}
u_{j}(x+iy)=\sqrt{y}\sum_{n\neq 0}\rho_{j}(n)K_{it_j}(2\pi|n|y)e(nx),
\end{equation*}
where $K_{\alpha}(x)$ is the $K$-Bessel function and $\rho_{j}(n)=\rho_{j}(1)\lambda_{j}(n).$  The coefficients $\{\lambda_{j}(n)\}$  satisfy (for $m,n\geq 1$) the Hecke identity
\begin{equation}\label{eq:multipFourcoeff2}
\lambda_{j}(n)\lambda_{j}(m)=\sum_{d|(m,n)}\lambda_{j}\left( \frac{nm}{d^2}\right).
\end{equation}
Define
\begin{equation*}
\alpha_j=\frac{|\rho_{j}(1)|^2}{\cosh(\pi t_j)}.
\end{equation*}
It is known \cite[(30)]{HM}, that   $t_j^{-\epsilon}\ll \alpha_j \ll t_j^{\epsilon}$. The  symmetric square $L$-function attached to $u_j$ is defined (for $\Re{s}>1$) by
\begin{equation*}
L(\sym^2 u_{j},s)=\zeta(2s)\sum_{n=1}^{\infty}\frac{\lambda_{j}(n^2)}{n^s}
\end{equation*}
and can be  analytically  continued to the whole complex plane. The following functional equation takes place:
\begin{equation}\label{func.eq}
L_{\infty}(s,t_j)L(\sym^2 u_{j},s)=L_{\infty}(1-s,t_j)L(\sym^2 u_{j},1-s),
\end{equation}
where
\begin{equation}\label{L.infinity}
L_{\infty}(s,t_j)=\pi^{-3s/2}\Gamma\left(\frac{s}{2}\right)\Gamma\left(\frac{s+2it_j}{2}\right)\Gamma\left(\frac{s-2it_j}{2}\right).
\end{equation}
Using  \eqref{func.eq}, we obtain the following approximate functional equation.
%%%%%%%%%%%%%%%%%%%%%%%%%%%%%%%%%%%%%%%%%%%%%%%%%%%%%%%%%%%%%%%%%%%%%%%%%%%%%%
\begin{lem}
We have
\begin{equation}\label{approx.func.eq.}
L(\sym^2 u_{j},1/2)=2\sum_{m=1}^{\infty}\frac{\lambda_j(m^2)}{m^{1/2}}V(m,t_j),
\end{equation}
where for any $y>0$ and $a>0$
\begin{equation}\label{approx.fun.eq.Vdef}
V(y,t_j)=\frac{1}{2\pi i}\int_{(a)}\frac{L_{\infty}(1/2+z,t_j)}{L_{\infty}(1/2,t_j)}\zeta(1+2z)G(z)y^{-z}\frac{dz}{z},
\end{equation}
and
\begin{equation}\label{Gdef}
G(z)=\exp(z^2)P_{2026}(z^2).
\end{equation}
Here $P_{2026}$ is a polynomial of degree 2026 such that $P_{2026}(0)=1$ and  $P_{2026}(1)=P_{2026}((1/2+2j)^2)=0$ for $j=0,\ldots, 2025$. Moreover, for any positive $y,\,t_j, A$ we have
\begin{equation}\label{Vestimate}
V(y,t_j)\ll\left(\frac{t_j}{y}\right)^{A},
\end{equation}
and for $1\le y\ll t_j^{1+\epsilon}$ the following asymptotic formula holds
\begin{multline}\label{Vapproximation}
V(y,t_j)=\frac{1}{2\pi i}\int_{(a)}
\left(\frac{t_j}{\pi^{3/2}y}\right)^z
\frac{\Gamma(1/4+z/2)}{\Gamma(1/4)}\zeta(1+2z)G(z)\\\times
\left(1+\sum_{k=1}^{N-1}\frac{p_{2k}(v)}{t_j^k}\right)\frac{dz}{z}+O(t_j^{-N+\epsilon}),
\end{multline}
where $v=\Im(z)$, and $p_n(v)$ is a polynomial of degree n.
\end{lem}
%%%%%%%%%%%%%%%%%%%%%%%%%%%%%%%%%%%%%%%%%%%%%%%%%%%%%%%%%%%%%%%%%%%%%%%%%%%%%%

%%%%%%%%%%%%%%%%%%%%%%%%%%%%%%%%%%%%%%%%%%%%%%%%%%%%%%%%%%%%%%%%%%%%%%%%%%%%%%%%%%%%%%%%%%%
%%%%%%%%%%%%%%%%%%%%%%%%%%%%%%%%%%%%%%%%%%%%%%%%%%%%%%%%%%%%%%%%%%%%%%%%%%%%%%%%%%%%%%%%%%%
\section{The first twisted moment}\label{sec:1st mom}
Our derivation of the new asymptotic formula for the second twisted moment is based on the following reciprocity-type formula for the first twisted moment:
\begin{equation}\label{1mom def}
\M_1(l;h):=\sum_{j}h(t_j)\alpha_{j}\lambda_{j}(l^2)L(\sym^2 u_{j},1/2).
\end{equation}
Here  $h(t)$ is a weight function satisfying the following conditions:
\begin{description}
  \item[C1] $h(t)$ is an even function;
  \item[C2] $h(t)$ is holomorphic in the strip $|\Im(t)|<\Upsilon$ for some $\Upsilon>1/2$;
  \item[C3] $h(t)$ is such that the following estimate $h(t)\ll(1+|t|)^{-2-\epsilon}$ holds in the strip $|\Im(t)|<\Upsilon$, $\Upsilon>1/2$;
  \item[C4] $h(\pm(n+1/2)i)=0$ for $n=0,1,\ldots N-1$, where $N>0$ is a sufficiently large integer.
\end{description}
The reciprocity formula expresses \eqref{1mom def} in terms of the moments of Zagier's $L$-series, which are defined (cf. \cite[Section 2]{SY}, \cite[Proposition 3]{Zag}) for $\Re(s) > 1$ as
\begin{equation}\label{Lbyk}
\Zag_{n}(s)=\frac{\zeta(2s)}{\zeta(s)}\sum_{q=1}^{\infty}\frac{\rho_q(n)}{q^{s}}=\zeta(2s)\sum_{q=1}^{\infty}\frac{\ups_q(n)}{q^{s}},
\end{equation}
where
\begin{equation}\label{rho upsilon def}
\rho_q(n):=\#\{x\Mod{2q}:x^2\equiv n\Mod{4q}\},\quad
\ups_q(n):=\sum_{q_2q_3=q}\mu(q_2)\rho_{q_3}(n).
%\lambda_q(n):=\sum_{q_{1}^{2}q_2q_3=q}\mu(q_2)\rho_{q_3}(n).
\end{equation}
The following is a special case of the main result of \cite{Bal} (see Theorem 12 and Lemmas 4, 7, and 9 therein).

%%%%%%%%%%%%%%%%%%%%%%%%%%%%%%%%%%%%%%%%%%%%%%%%%%%%%%%%%%%%%%%%%%%%%%%%%%%%%%

\begin{thm}\label{thm rho=1/2 exact} For  any function $h(t)$  satisfying the conditions $(C1)-(C4)$ we have
\begin{equation}\label{eq:M1lrho=1/2}
\M_1(l;h(\cdot))=MT(l;h)+CT(l;h)+ET(l;h)+S_1(l;h)+S_2(l;h),
\end{equation}
where
\begin{multline}\label{M1 MT}
MT(l;h)=\frac{1}{2\pi^2l^{1/2}}\int_{-\infty}^{\infty}th(t)\tanh(\pi t)\Biggl(
3\gamma-\frac{\pi}{2}-2\log l-3\log(2\pi)\\+
\psi(1/4+it)+\psi(1/4-it)\Biggr)dt,
\end{multline}
\begin{equation}\label{M1 CT}
CT(l;h)=-\frac{\zeta(1/2)}{\pi}\int_{-\infty}^{\infty}\frac{\tau_{it}(l^2)h(t)}{|\zeta(1+2it)|^2}\zeta(1/2+2it)\zeta(1/2-2it)dt,
\end{equation}
\begin{equation}\label{M1 ET}
ET(l;h)=-\frac{2\zeta(0)}{\zeta(3/2)}\tau_{1/4}(l^2)h\left(\frac{1}{4i}\right)-
\frac{\mathscr{L}_{-4l^2}(1/2)}{i(2 l\pi^3)^{1/2}}\int_{-\infty}^{\infty}\frac{th(t)}{\cosh(\pi t)}\frac{\Gamma(1/4+it)}{\Gamma(3/4+it)}dt,
\end{equation}
\begin{equation}\label{M1 S12}
S_1(l;h)=\sum_{n=1}^{2l-1}\frac{\Zag_{(2l-n)^2-4l^2}(1/2)}{\sqrt{2\pi(2l- n)}}I\left(2-\frac{n}{l};h\right),\quad
S_2(l;h)=\sum_{n=1}^{\infty}\frac{\Zag_{(n+2l)^2-4l^2}(1/2)}{\sqrt{2\pi(n+2l)}}I\left(2+\frac{n}{l};h\right),
\end{equation}
and for  $x \geq 2$
\begin{multline}\label{integralIgeq2}
I(x;h):=\frac{2^{3/2}i}{\pi^{3/2}}\int_{-\infty}^{\infty}\frac{th(t)}{\cosh(\pi t)}\left(\frac{2}{x}\right)^{2it}
\frac{\Gamma(1/4+it)\Gamma(3/4+it)}{\Gamma(1+2it)}\\ \times \sin\left( \pi(1/4-it)\right)
{}_2F_{1}\left(1/4+it,3/4+it,1+2it;\frac{4}{x^2} \right)dt,
\end{multline}
while  for $0<x<2$
\begin{multline}\label{integralIleq2}
I(x;h):=\frac{2i}{\pi^{3/2}}\int_{-\infty}^{\infty}\frac{th(t)}{\cosh(\pi t)}x^{1/2}
\frac{\Gamma(1/4+it)\Gamma(1/4-it)}{\Gamma(1/2)}\\ \times \cos\left( \pi(1/4+it)\right)
{}_2F_{1}\left(1/4+it,1/4-it,1/2;\frac{x^2}{4} \right)dt.
\end{multline}
\end{thm}

Furthermore, by \cite[(21), (23)]{Bal}, we have
\begin{equation}\label{eq:integralI}
I(x;h)=\frac{2i}{\pi}\int_{-\infty}^{\infty}\frac{th(t)}{\cosh(\pi t)}I_t(x)dt,
\end{equation}
where $-1-2N<\Delta<1/2$, and
\begin{equation}\label{I_t(x)def}
I_t(x)=\frac{1}{2\pi i}\int_{(\Delta)}\frac{\Gamma(w/2+it)}{\Gamma(1-w/2+it)}\Gamma(1/2-w)\sin\left( \pi \frac{1/2+w}{2}\right)x^wdw.
\end{equation}

%%%%%%%%%%%%%%%%%%%%%%%%%%%%%%%%%%%%%%%%%%%%%%%%%%%%%%%%%%%%%%%%%%%%%%%%%%%%%%%%%%%%%%%%%%%
%%%%%%%%%%%%%%%%%%%%%%%%%%%%%%%%%%%%%%%%%%%%%%%%%%%%%%%%%%%%%%%%%%%%%%%%%%%%%%%%%%%%%%%%%%%
\section{The second twisted moment I}\label{sec:2mom to 1mom}
Let
\begin{equation}\label{2mom def}
\M_2(r;h(\cdot)):=\sum_{j}h(t_j)\alpha_{j}\lambda_{j}(r^2)L(\sym^2 u_{j},1/2)^2.
\end{equation}
Expressing one of the $L$-function via the approximate functional equation  \eqref{approx.func.eq.} and applying \eqref{eq:multipFourcoeff2}, we get
\begin{equation}\label{M2 to M1 eq1}
\M_2(r;h(\cdot))=2\sum_{m=1}^{\infty}\frac{1}{\sqrt{m}}\sum_{e|(m^2,r^2)}\M_1\left(\frac{mr}{e};h(\cdot)V(m,\cdot)\right).
\end{equation}
Throughout the paper, we will use the decomposition $e = e_1 e_2^2$, where $e_1$ is square-free. Under this convention, $e \mid m^2$ if and only if $e_1 e_2 \mid m$. Therefore,
\begin{equation}\label{M2 to M1 eq2}
\M_2(r;h(\cdot))=2\sum_{e|r^2}\sum_{m=1}^{\infty}\frac{1}{\sqrt{me_1e_2}}\M_1\left(\frac{mr}{e_2};h(\cdot)V(me_1e_2,\cdot)\right).
\end{equation}
Applying \eqref{eq:M1lrho=1/2}, we obtain
\begin{equation}\label{M2 to M1 eq3}
\M_2(r;h(\cdot))=\MT_2(r)+\MTn(r)+\SZE(r)+\SIN(r),
\end{equation}
where
\begin{equation}\label{2mom MT2def}
\MT_2(r;h)=2\sum_{e|r^2}\sum_{m=1}^{\infty}\frac{1}{\sqrt{me_1e_2}}MT\left(\frac{mr}{e_2},h(\cdot)V(me_1e_2,\cdot)\right),
\end{equation}
\begin{equation}\label{2mom MTn def}
\MTn(r;h)=2\sum_{e|r^2}\sum_{m=1}^{\infty}\frac{1}{\sqrt{me_1e_2}}\left(
CT\left(\frac{mr}{e_2},h(\cdot)V(me_1e_2,\cdot)\right)+ET\left(\frac{mr}{e_2},h(\cdot)V(me_1e_2,\cdot)\right)\right),
\end{equation}
\begin{multline}\label{2mom SZE def}
\SZE(r;h)=\frac{2}{\sqrt{2\pi}}\sum_{e|r^2}\sum_{m=1}^{\infty}\frac{1}{\sqrt{me_1e_2}}
\sum_{n=1}^{2mr/e_2-1}\frac{\Zag_{-n(4mr/e_2-n)}(1/2)}{\sqrt{2mr/e_2- n}}\\\times
I\left(2-\frac{ne_2}{mr};h(\cdot)V(me_1e_2,\cdot)\right),
\end{multline}
\begin{equation}\label{2mom SIN def}
\SIN(r;h)=\frac{2}{\sqrt{2\pi}}\sum_{e|r^2}\sum_{m=1}^{\infty}\frac{1}{\sqrt{me_1e_2}}
\sum_{n=1}^{\infty}\frac{\Zag_{n(n+4mr/e_2)}(1/2)}{\sqrt{n+2mr/e_2}}I\left(2+\frac{ne_2}{mr};h(\cdot)V(me_1e_2,\cdot)\right).
\end{equation}
It follows from \eqref{M1 CT}, \eqref{M1 ET} and \cite[Lemma 5.3]{BF2mom} that for $G\ll T^{1-\epsilon}$ one has
\begin{equation}\label{2mom MTn est}
\MTn(r;h)\ll T^{\epsilon}\max(G,T^{2/3}).
\end{equation}
To simplify the notation (cf. \cite[(4.3)]{Frolsym2}), let
\begin{equation}\label{I(m,x) def}
I(m,x):=I\left(x;h(\cdot)V(m,\cdot)\right).
\end{equation}

%%%%%%%%%%%%%%%%%%%%%%%%%%%%%%%%%%%%%%%%%%%%%%%%%%%%%%%%%%%%%%%%%%%%%%%%%%%%%%%%%%%%%%%%%%%
%%%%%%%%%%%%%%%%%%%%%%%%%%%%%%%%%%%%%%%%%%%%%%%%%%%%%%%%%%%%%%%%%%%%%%%%%%%%%%%%%%%%%%%%%%%
\section{Evaluation of $\MT_2(r;h)$}\label{sec: main term diagonal}
It follows from \cite[(33), (34)]{Bal} that
\begin{multline}\label{M1 MT eq2}
MT(l;h)=\frac{1}{\pi^2}\int_{-\infty}^{\infty}th(t)\tanh(\pi t)
\lim_{u\rightarrow0}\Biggl(
\frac{\zeta(1+2u)}{l^{1/2+u}}+
\frac{\zeta(1-2u)}{l^{1/2-u}}
\Gamma(1/2-u)\\\times 
\frac{2^{1/2+u}\pi^{-1/2+3u}i}{\sinh(\pi t)}
\sin\left( \pi(1/4+u/2-it)\right)\frac{\Gamma(1/4-u/2+it)\Gamma(3/4-u/2+it)}{\Gamma(1/4+u/2+it)\Gamma(3/4+u/2+it)}
\Biggr)dt.
\end{multline}
Let us introduce the notation
\begin{multline}\label{MT_1(u,t)}
MT_1(u,t):=
\frac{2^{1/2+u}\pi^{-1/2+3u}i}{\sinh(\pi t)}
\sin\left( \pi(1/4+u/2-it)\right)\Gamma(1/2-u)\\\times\frac{\Gamma(1/4-u/2+it)\Gamma(3/4-u/2+it)}{\Gamma(1/4+u/2+it)\Gamma(3/4+u/2+it)}.
\end{multline}
It follows from \eqref{2mom MT2def}, \eqref{M1 MT eq2}, \eqref{MT_1(u,t)} and \eqref{approx.fun.eq.Vdef} that
\begin{multline}\label{2mom MT2 eq1}
\MT_2(r;h)=\frac{2}{\pi^2}\int_{-\infty}^{\infty}th(t)\tanh(\pi t)
\sum_{e|r^2}\frac{1}{\sqrt{re_1}}\frac{1}{2\pi i}\int_{(a)}\frac{L_{\infty}(1/2+z,t)}{L_{\infty}(1/2,t)}\\\times
\frac{\zeta(1+2z)G(z)}{(e_1e_2)^z}
\lim_{u\rightarrow0}\left(
\frac{\zeta(1+2u)\zeta(1+z+u)}{(r/e_2)^{u}}+
\frac{\zeta(1-2u)\zeta(1+z-u)}{(r/e_2)^{-u}}MT_1(u,t)
\right)\frac{dz}{z}dt.
\end{multline}
Note that \eqref{2mom MT2 eq1} is the Maass form analogue of \cite[(7.5)]{BF2momWA2026}. Evaluating the limit via L'H\^opital's rule, we obtain
\begin{multline}\label{2mom MT2 lim}
\lim_{u\to0}\biggl(\frac{\zeta(1+2u)\zeta(1+z+u)}{(r/e_2)^{u}}+
\frac{\zeta(1-2u)\zeta(1+z-u)}{(r/e_2)^{-u}}MT_1(u,t) \biggr)=\\=
\zeta'(1+z)+
\zeta(1+z)\biggl(2\gamma-\log\frac{r}{e_2}-\frac{1}{2}\frac{d}{du}MT_1(u,t)\bigl|_{u=0}\biggr).
\end{multline}
One has
\begin{multline}\label{d/du MT_1(u,t)}
\frac{d}{du}MT_1(u,t)\bigl|_{u=0}=
\frac{2^{1/2}i}{\sinh(\pi t)}
\sin\left( \pi(1/4-it)\right)\Bigl(\log2+3\log\pi-\psi(1/4+it)-\\-\psi(3/4+it)-\psi(1/2)\Bigr)+
\frac{\pi i}{2^{1/2}\sinh(\pi t)}\cos\left( \pi(1/4-it)\right).
\end{multline}
It follows from \cite[5.5.4]{HMF} that
\begin{equation}
\psi(3/4+it)=\psi(1/4-it)+\frac{\pi}{\tan\pi(1/4-it)},
\end{equation}
and therefore, by \cite[5.4.13]{HMF},
\begin{multline}\label{d/du MT_1(u,t)2}
\frac{d}{du}MT_1(u,t)\bigl|_{u=0}=
\frac{2^{1/2}i}{\sinh(\pi t)}
\sin\left( \pi(1/4-it)\right)\biggl(3\log2+3\log\pi-\\-\psi(1/4+it)-\psi(1/4-it)+\gamma\biggr)-
\frac{\pi i}{2^{1/2}\sinh(\pi t)}\cos\left( \pi(1/4-it)\right).
\end{multline}
Since $h(t)$ is even, it follows that 
\begin{equation}
\int_{-\infty}^{\infty} t h(t) \tanh(\pi t) f(t) \, dt = 0
\end{equation}
whenever $f(t)$ is an odd function. Since
\begin{equation}
\frac{\sin\left( \pi(1/4-it)\right)}{\sinh(\pi t)}=\frac{\cosh\left( \pi t\right)}{\sqrt{2}\sinh(\pi t)}-\frac{i}{\sqrt{2}},\,
\frac{\cos\left( \pi(1/4-it)\right)}{\sinh(\pi t)}=\frac{\cosh\left( \pi t\right)}{\sqrt{2}\sinh(\pi t)}+\frac{i}{\sqrt{2}},
\end{equation}
it follows from \eqref{2mom MT2 lim} and  \eqref{d/du MT_1(u,t)2} that
\begin{multline}\label{2mom MT2 eq3}
\MT_2(r;h)=\frac{1}{\pi^2}\int_{-\infty}^{\infty}th(t)\tanh(\pi t)
\sum_{e|r^2}\frac{1}{\sqrt{re_1}}\frac{1}{2\pi i}\int_{(a)}\frac{L_{\infty}(1/2+z,t)}{L_{\infty}(1/2,t)}
\frac{\zeta(1+2z)G(z)}{(e_1e_2)^z}\\\times
\Bigl(2\zeta'(1+z)+
\zeta(1+z)\biggl(3\gamma-2\log\frac{r}{e_2}-3\log(2\pi)-\frac{\pi}{2}+\psi(1/4+it)+\psi(1/4-it)\biggr)\Bigr)
\frac{dz}{z}dt,
\end{multline}
which serves as an analogue of \cite[(7.7)]{BF2momWA2026}. Let
\begin{equation}\label{C1 def}
\Cc_1(z):=\log\frac{t^2}{8\pi^3}+2\frac{\zeta'(1+z)}{\zeta(1+z)}+3\gamma-\frac{\pi}{2},
\end{equation}
\begin{equation}\label{MT sum e def}
\sum_{e_1e_2|r}\frac{|\mu(e_1)|}{(e_1e_2)^{z}\sqrt{e_1}}\left(\Cc_1(z)+2\log\frac{e_2}{r}\right)=:r^{-z}\MAP(z,r).
\end{equation}
Proceeding as in \cite{BF2momWA2026} and applying \eqref{Vapproximation}, we finally obtain
\begin{multline}\label{2mom MT2 eq4}
\MT_2(r;h)=\frac{2}{\pi^2\sqrt{r}}\int_{0}^{\infty}th(t)\tanh(\pi t)
\frac{1}{2\pi i}\int_{(a)}\pi^{-3/2z}\frac{\Gamma(1/4+z/2)}{\Gamma(1/4)}\zeta(1+z)\\\times\zeta(1+2z)G(z)
\MAP(z,r)\left(\frac{t}{r}\right)^{z}\frac{dz}{z}dt+O\left(\frac{T^{\epsilon}}{\sqrt{r}}\right).
\end{multline}
%%%%%%%%%%%%%%%%%%%%%%%%%%%%%%%%%%%%%%%%%%%%%%%%%%%%%%%%%%%%%%%%%%%%%%%%%%%%%%%%%%%%%%%%%%%
%%%%%%%%%%%%%%%%%%%%%%%%%%%%%%%%%%%%%%%%%%%%%%%%%%%%%%%%%%%%%%%%%%%%%%%%%%%%%%%%%%%%%%%%%%%
\section{The Voronoi summation formula for Zagier's $L$-series}\label{sec:Voronoi}
After applying suitable transformations, we evaluate \eqref{2mom SZE def} and \eqref{2mom SIN def} by utilizing the Voronoi summation formula for Zagier's $L$-series established in \cite{BFVoron}. To this end, let (cf. \cite[Lemma 4.1]{BF2momWA2026})
\begin{equation}\label{Voronoi MTeq0}
\Resid(\phi;x)=\int_0^{\infty}\frac{\phi(y)}{x\sqrt{2y}}\left(\log\frac{2y}{\pi x^2}-\frac{\pi}{2}+3\gamma\right)dy+
\int_0^{\infty}\frac{\phi(-y)}{x\sqrt{2y}}\left(\log\frac{2y}{\pi x^2}+\frac{\pi}{2}+3\gamma\right)dy,
\end{equation}
and (cf. \cite[(1.11), (1.12)]{BFVoron})
\begin{equation}\label{phi hat+ to Phipm def0}
\widehat{\phi}(y):=\int_0^{\infty}\left(\frac{\phi(x)}{x}\Phi^{(+,+)}(xy)+\frac{\phi(-x)}{x}\Phi^{(+,-)}(xy)\right)dx,
\end{equation}
\begin{equation}\label{phi hat- to Phipm def0}
\widehat{\phi}(-y):=\int_0^{\infty}\left(\frac{\phi(x)}{x}\Phi^{(-,+)}(xy)+\frac{\phi(-x)}{x}\Phi^{(-,-)}(xy)\right)dx,
\end{equation}
where $y>0$. The kernel functions $\Phi^{(+,+)}(x)$ and $\Phi^{(-,-)}(x)$ are defined by (cf. \cite[(1.9)]{BFVoron}, \cite[(4.15)]{BF2momWA2026})
\begin{equation}\label{Phi++--def}
\Phi^{(\pm,\pm)}(x)=
-\frac{\sqrt{x}}{\sqrt{2}}\left(Y_0(2\sqrt{x})\mp J_{0}(2\sqrt{x})\right).
\end{equation}
The functions $\Phi^{(+,-)}(x)$ and $\Phi^{(-,+)}(x)$ are given by (cf. \cite[(1.10)]{BFVoron})
\begin{equation}\label{Phi+--+def}
\Phi^{(\pm,\mp)}(x):=\frac{2\sqrt{x}K_{0}(2\sqrt{x})}{\Gamma^2(1/2\mp 1/4)}.
\end{equation}
As usual,  $\left(\frac{c}{d}\right)$ denotes the extended Jacobi symbol (cf. \cite[Sec.~A1]{Biro2000}), and $\epsilon_d$ is defined by
\begin{equation}\label{epsilon def}
\epsilon_{q} =
\begin{cases}
1, & \text{if } q \equiv 1 \pmod{4}, \\
i, & \text{if } q \equiv 3 \pmod{4}.
\end{cases}
\end{equation}

Finally, we record the Voronoi summation formula from \cite[Theorems 1.1, 1.2, and 1.4]{BFVoron}, where $\sgn(\phi) = \pm 1$ depending on whether the support of $\phi$ is contained in $\mathbb{R}_{\pm}$.

%%%%%%%%%%%%%%%%%%%%%%%%%%%%%%%%%%%%%%%%%%%%%%%%%%%%%5
\begin{lem}\label{Thm Voronoi an c0mod4}
For $c\equiv0\Mod{4}$ and $ad\equiv1\Mod{c}$ we have
\begin{multline}\label{Thm.eq Voronoi an c0mod4}
\sum_{n=-\infty}^{\infty}\frac{\Zag_{n}\left(1/2\right)}{\sqrt{|n|}}\phi(n)e\left( \frac{an}{c}\right)=
\TM(M_0)e(1/8)\Resid(\phi;c)+\\+
\TM(M_0)e(1/8)\sum_{n\neq0}\frac{\Gamma\left(\frac{1}{2}-\frac{\sgn{n}}{4}\right)}{\Gamma\left(\frac{1}{2}-\frac{\sgn{\phi}}{4}\right)}
\frac{\Zag_{n}\left(1/2\right)}{\sqrt{|n|}}\widehat{\phi}\left(\frac{4\pi^2n}{c^2}\right)e\left( -\frac{dn}{c}\right),
\end{multline}
where $\TM(M_0)=\bar{\epsilon}_{d}\left(\frac{c}{d}\right)$.
\end{lem}
%%%%%%%%%%%%%%%%%%%%%%%%%%%%%%%%%%%%%%%%%%%%%5
\begin{lem}\label{Thm Voronoi an codd}
For $(c,2)=1$ and $4ad\equiv -1\Mod{c}$ we have
\begin{multline}\label{Thm.eq Voronoi an codd}
\sum_{n=-\infty}^{\infty}\frac{\Zag_{n}\left(1/2\right)}{\sqrt{|n|}}\phi(n)e\left( \frac{an}{c}\right)=
\TM(M_1)\sqrt{2}\Resid(\phi;4c)+\\+
\TM(M_1)\sqrt{2}\sum_{n\neq0}
\frac{\Gamma\left(\frac{1}{2}-\frac{\sgn{n}}{4}\right)}{\Gamma\left(\frac{1}{2}-\frac{\sgn{\phi}}{4}\right)}
\frac{\Zag_{4n}\left(1/2\right)}{\sqrt{|4n|}}
\widehat{\phi}\left(\frac{\pi^2n}{c^2}\right)e\left(\frac{dn}{c}\right),
\end{multline}
where $\TM(M_1)=\bar{\epsilon}_{c}\left(\frac{4d}{c}\right)$.
\end{lem}

%%%%%%%%%%%%%%%%%%%%%%%%%%%%%%%%%%%%%%%%%%%%%%%%%%%%%%%%%%%%%%%%%%%%%%%%%%%%%%%%%

\begin{lem}\label{Thm Voronoi an c2mod4}
For  $c\equiv2\Mod{4}$, let $c_1=c/2$. For $2ad_5\equiv-1\Mod{c_1}$,  $8ad_4\equiv-1\Mod{c_1}$ we have
\begin{multline}\label{Thm.eq Voronoi an c2mod4}
\sum_{n=-\infty}^{\infty}\frac{\Zag_{n}\left(1/2\right)}{\sqrt{|n|}}\phi(n)e\left( \frac{an}{c}\right)=
\frac{\sqrt{2}}{\TM(M_4)}\Resid(\phi;4c_1)-
\TM(M_5)\sqrt{2}\Resid(\phi;4c_1)\\+
\frac{\sqrt{2}}{\TM(M_4)}\sum_{n\neq0}\frac{\Gamma\left(\frac{1}{2}-\frac{\sgn{n}}{4}\right)}{\Gamma\left(\frac{1}{2}-\frac{\sgn{\phi}}{4}\right)}
\frac{\Zag_{n}\left(1/2\right)}{\sqrt{|n|}}
\widehat{\phi}\left(\frac{\pi^2n}{4c_1^2}\right)e\left(\frac{d_4n}{c_1}\right)-\\-
\TM(M_5)\sqrt{2}\sum_{n\neq0}\frac{\Gamma\left(\frac{1}{2}-\frac{\sgn{n}}{4}\right)}{\Gamma\left(\frac{1}{2}-\frac{\sgn{\phi}}{4}\right)}
\frac{\Zag_{4n}\left(1/2\right)}{\sqrt{|4n|}}
\widehat{\phi}\left(\frac{\pi^2n}{c_1^2}\right)e\left(\frac{d_5n}{c_1}\right),
\end{multline}
where $\TM^{-1}(M_4)=\epsilon_{c_1}\left(\frac{8a}{c_1}\right)$ and $\TM(M_5)=\bar{\epsilon}_{c_1}\left(\frac{4d_5}{c_1}\right)$.
\end{lem}

%%%%%%%%%%%%%%%%%%%%%%%%%%%%%%%%%%%%%%%%%%%%%%%%%%%%%%%%%%%%%%%%%%%%%%%%%%%%%%
To estimate the integral transforms on the right-hand side of the Voronoi formula, we utilize the following classical result \cite[Lemma 8.1]{BKY} (see also \cite[Lemma A.1]{AHLQ}).

\begin{lem}\label{Lemma BKY}
Suppose that there exist positive parameters $R, P, X, Y$, and $V$ such that
\begin{equation}\label{BKYconditions}
|f'(x)|\gg R,\quad
f^{(i)}(x)\ll\frac{Y}{P^i},\quad
g^{(j)}(x)\ll\frac{X}{V^j}
\end{equation}
for $i\ge1,j\ge0$. Then
\begin{equation}\label{I BKY est}
I=\int_a^{b}g(x)e(f(x))dx\ll(b-a)X\left(\frac{1}{RV}+\frac{1}{RP}+\frac{Y}{R^2P^2}\right)^{A}.
\end{equation}
\end{lem}

%%%%%%%%%%%%%%%%%%%%%%%%%%%%%%%%%%%%%%%%%%%%%%%%%%%%%%%%%%%%%%%%%%%%%%%%%%%%%%%%%%%%%%%%%%%
%%%%%%%%%%%%%%%%%%%%%%%%%%%%%%%%%%%%%%%%%%%%%%%%%%%%%%%%%%%%%%%%%%%%%%%%%%%%%%%%%%%%%%%%%%%
\section{Evaluation of $\SIN(r;h)$}\label{sec:SIN}
We begin our investigation of \eqref{2mom SIN def} by studying the asymptotic properties of $I\left(me_1e_2, 2+\frac{ne_2}{mr}\right)$ (cf. \eqref{I(m,x) def}). In the case where
\begin{equation}
\frac{ne_2}{mr} \gg T^{-2+\epsilon},
\end{equation}
this analysis was carried out in \cite[Section 4]{Frolsym2}, where it was shown that the contribution of terms satisfying $\frac{ne_2}{mr} \gg T^{\epsilon}G^{-2}$ is negligible. Furthermore, within the range
\begin{equation}
T^{-2+\epsilon} \ll \frac{ne_2}{mr} \ll T^{\epsilon}G^{-2},
\end{equation}
one can replace the function by its main term given in \cite[(4.11)]{Frolsym2}:
\begin{equation}\label{2mom SIN I eq1}
I\left(me_1e_2,x\right)=\left(\frac{x^2}{x^2-4}\right)^{1/4}
GT^{1/2}\exp(-2iTA(x))\exp\left(-G^2A(x)^2\right)V(me_1e_2,T),
\end{equation}
where $A(x):=\arcosh{x/2}.$ 

If $x \ll T^{-2+\epsilon}$, we are unable to obtain an asymptotic expansion for $I(m,x)$, as the argument of the confluent hypergeometric functions in \cite[Lemma 4.4]{BF2mom} remains bounded by $T^{\epsilon}$. To estimate $I(m,x)$, we must first bound the hypergeometric function featured in \cite[Lemma 4.4]{BF2mom}. For this purpose, note that by \cite[(13.6.10)]{HMF},
\begin{equation}\label{U to K0}
U(1/2,1,z)=\frac{e^{z/2}}{\sqrt{\pi}}K_0(z/2).
\end{equation}
Using \cite[(13.3.27),(10.29.3)]{HMF} and  \eqref{U to K0}, we have
\begin{multline}\label{U1/2 2 to K eq1}
U(1/2,2,z)=-e^z\frac{d}{dz}\left(e^{-z}U(1/2,1,z)\right)=-e^z\frac{d}{dz}\left(\frac{e^{-z/2}}{\sqrt{\pi}}K_0(z/2)\right)=\\=
\frac{-e^z}{\sqrt{\pi}}\left(-\frac{e^{-z/2}}{2}K_0(z/2)-\frac{e^{-z/2}}{2}K_1(z/2)\right)=
\frac{e^{z/2}}{2\sqrt{\pi}}\left(K_0(z/2)+K_1(z/2)\right).
\end{multline}
It follows from \cite[Lemma 4.4]{BF2mom}, \eqref{U to K0}, \eqref{U1/2 2 to K eq1} and \cite[(10.30.2),(10.30.3)]{HMF} that for $t/z\ll t^{\epsilon}$
\begin{equation}\label{2F1(1/2,1/2,1+2ir)est}
\frac{\Gamma(1/2+2it)}{\Gamma(1+2it)}{}_2F_{1}\left( \frac{1}{2},\frac{1}{2},1+2it; -z\right)\ll
z^{-1/2}|\log(t/z)|.
\end{equation}
In \cite[Lemma 4.6]{BF2mom}, the result of \cite[Lemma 4.4]{BF2mom} is applied with $z = \frac{x-\sqrt{x^2-4}}{2\sqrt{x^2-4}} \asymp (x-2)^{-1/2}$ as $x$ is close to $2$. Combining \eqref{2F1(1/2,1/2,1+2ir)est} with the final equation in \cite[Lemma 4.6]{BF2mom}, we obtain
\begin{multline}\label{2F1(1/4+ir,3/4+ir,1+2ir)est}
{}_2F_{1}\left( \frac{1}{4}+it,\frac{3}{4}+it,1+2it; \frac{4}{x^2}\right)\ll \sqrt{t}
\left(\frac{x^2}{x^2-4}\right)^{1/4}(x-2)^{1/4}\log(t\sqrt{x-2})\ll\\\ll \sqrt{t}\log(t\sqrt{x-2}).
\end{multline}
Combining \eqref{2F1(1/4+ir,3/4+ir,1+2ir)est}, \eqref{I(m,x) def}, and \eqref{integralIgeq2}, we prove that  for $0<x-2\ll t^{-2+\epsilon}$
\begin{equation}\label{I(m,x)x=2 est}
I\left(me_1e_2,x\right)\ll
\int_{-\infty}^{\infty}th(t)|\log(t\sqrt{x-2})V(me_1e_2,t)|dt\ll T^{1+\epsilon}G|\log(T\sqrt{x-2})|.
\end{equation}
Since the behavior of $I\left(me_1e_2, x\right)$ is governed by two distinct regimes, we split \eqref{2mom SIN def} into two subsums as follows:
\begin{equation}\label{SIN=Sigma21+Sigma22}
\SIN(r;h)=\Sigma_{2,1}(r;h)+\Sigma_{2,2}(r;h)+O(T^{-A}),
\end{equation}
where in $\Sigma_{2,1}$ the $n$-sum is over
\begin{equation}
0<n\ll\frac{mr}{e_2}T^{\epsilon-2},
\end{equation}
and in $\Sigma_{2,2}$ the $n$-sum is over
\begin{equation}
\frac{mr}{e_2}T^{\epsilon-2}\ll n\ll\frac{mrT^{\epsilon}}{e_2G^2}.
\end{equation}
The summands corresponding to $n \gg \frac{mr T^{\epsilon}}{e_2 G^2}$ contribute an error of $O(T^{-A})$ according to the preceding discussion. Furthermore, we can truncate the $m$-sum to the range $m \ll T^{1+\epsilon}/(e_1 e_2)$ by virtue of \eqref{Vestimate}. Initially, we treat both sums in \eqref{SIN=Sigma21+Sigma22} in an identical manner. Proceeding similarly to \cite[Section 8]{BF2momWA2026}, we perform the change of variables
$4mr/e_2+n=q$, i.e.
\begin{equation}\label{SIN m to q}
m=\frac{q-n}{4re_2^{-1}},\quad q\equiv n\Mod{\frac{4r}{e_2}}.
\end{equation}
In order to do this, we first change the order of summation to make the sum over $n$ the outer one; thus, in $\Sigma_{2,1}$ and $\Sigma_{2,2}$, respectively, we have
\begin{equation}\label{Sigma2122 n m sums}
\sum_{n\ll \frac{rT^{\epsilon-1}}{e}}\sum_{ne_2T^{2-\epsilon}/r\ll m\ll T^{1+\epsilon}/e_1e_2},\quad
\sum_{n\ll \frac{rT^{1+\epsilon}}{G^{2}e}}\sum_{ne_2G^{2-\epsilon}/r\ll m\ll \min(T^{1+\epsilon}/e_1e_2,ne_2T^{2-\epsilon}/r)}.
\end{equation}
Therefore, after applying the change of variables \eqref{SIN m to q}, we finally obtain
\begin{equation}\label{Sigma21 eq1}
\Sigma_{2,1}(r;h)=\frac{4}{\sqrt{\pi}}\sum_{e|r^2}\frac{\sqrt{r}}{\sqrt{e}}\sum_{n\ll \frac{rT^{-1+\epsilon}}{e}}
\sum_{\substack{nT^{2-\epsilon}\ll q\ll rT^{1+\epsilon}/e\\q\equiv n\Mod{4r/e_2}}}
\frac{\Zag_{qn}(1/2)}{\sqrt{q^2-n^2}}I\left(e\frac{q-n}{4r}, 2\frac{q+n}{q-n}\right),
\end{equation}
\begin{equation}\label{Sigma22 eq1}
\Sigma_{2,2}(r;h)=\frac{4}{\sqrt{\pi}}\sum_{e|r^2}\frac{\sqrt{r}}{\sqrt{e}}\sum_{n\ll \frac{rT^{1+\epsilon}}{eG^2}}
\sum_{\substack{nG^{2-\epsilon}\ll q\ll \min(rT^{1+\epsilon}/e,nT^{2-\epsilon}) \\q\equiv n\Mod{4r/e_2}}}
\frac{\Zag_{qn}(1/2)}{\sqrt{q^2-n^2}}I\left(e\frac{q-n}{4r}, 2\frac{q+n}{q-n}\right).
\end{equation}
Performing an additional change of variables $q = l/n$, which introduces the congruence condition $l \equiv n^2 \pmod{4rn/e_2}$, and detecting this congruence condition via
\begin{equation}\label{deltaq(l)}
\delta_{4rn/e_2}(l-n^2)=\frac{e_2}{4rn}\sum_{c|4rn/e_2}
\mathop{{\sum}^*}_{a \Mod{c}}e\left(\frac{a(l-n^2)}{c}\right),
\end{equation}
we get
\begin{multline}\label{Sigma21 eq2}
\Sigma_{2,1}(r;h)=\frac{1}{\sqrt{\pi}}\sum_{e|r^2}\frac{1}{\sqrt{re_1}}\sum_{n\ll \frac{rT^{-1+\epsilon}}{e}}\sum_{c|4rn/e_2}
\mathop{{\sum}^*}_{a \Mod{c}}e\left(\frac{-an^2}{c}\right)\\
\sum_{n^2T^{2-\epsilon}\ll l\ll nrT^{1+\epsilon}/e}
\frac{\Zag_{l}(1/2)}{\sqrt{l^2-n^4}}
I\left(e\frac{l-n^2}{4rn}, 2\frac{l+n^2}{l-n^2}\right)e\left(\frac{al}{c}\right),
\end{multline}
\begin{multline}\label{Sigma22 eq2}
\Sigma_{2,2}(r;h)=
\frac{1}{\sqrt{\pi}}\sum_{e|r^2}\frac{1}{\sqrt{re_1}}\sum_{n\ll \frac{rT^{1+\epsilon}}{eG^2}}\sum_{c|4rn/e_2}
\mathop{{\sum}^*}_{a \Mod{c}}e\left(\frac{-an^2}{c}\right)\\
\sum_{n^2G^{2-\epsilon}\ll l\ll \min(nrT^{1+\epsilon}/e,n^2T^{2-\epsilon})}
\frac{\Zag_{l}(1/2)}{\sqrt{l^2-n^4}}
I\left(e\frac{l-n^2}{4rn}, 2\frac{l+n^2}{l-n^2}\right)e\left(\frac{al}{c}\right).
\end{multline}
%%%%%%%%%%%%%%%
Let us first consider \eqref{Sigma22 eq2}. In that case we perform a smooth partition of unity in $l$ variable and apply \eqref{2mom SIN I eq1}, obtaining
\begin{multline}\label{Sigma22 eq3}
\Sigma_{2,2}(r;h)=
\frac{G\sqrt{T}}{\sqrt{2\pi}}\sum_{e|r^2}\frac{1}{\sqrt{re_1}}\sum_{n\ll\frac{rT^{1+\epsilon}}{eG^2}}\frac{1}{\sqrt{n}}\sum_{c|4rn/e_2}
\mathop{{\sum}^*}_{a \Mod{c}}e\left(\frac{-an^2}{c}\right)\\
\sum_{n^2G^{2-\epsilon}\ll L\ll \min(nrT^{1+\epsilon}/e,n^2T^{2-\epsilon})}\sum_{l=-\infty}^{\infty}
\frac{\Zag_{l}(1/2)}{\sqrt{l}}g_2(l)e\left(\frac{al}{c}\right),
\end{multline}
where
\begin{equation}\label{g2 def}
g_2(l)=\frac{l^{1/4}}{\sqrt{l-n^2}}U\left(\frac{l}{L}\right)
\exp(-2iTA(l))\exp\left(-G^2A(l)^2\right)V\left(e\frac{l-n^2}{4rn},T\right),
\end{equation}
with $A(l)=\arcosh\frac{l+n^2}{l-n^2}$. The next step is to apply the Voronoi summation formula (see Section \ref{sec:Voronoi}) to the sum over $l$ in \eqref{Sigma22 eq3}. To this end, we split the sum over $c$ in \eqref{Sigma22 eq3} into three cases:
\begin{equation}\label{c cases}
c\equiv0\Mod{4},\quad (c,2)=1,\quad c\equiv2\Mod{4}.
\end{equation}
As a result,
\begin{equation}\label{Sigma22 toSigma22j}
\Sigma_{2,2}(r;h)=
\frac{G\sqrt{T}}{\sqrt{2\pi}}\sum_{e|r^2}\frac{1}{\sqrt{re_1}}\sum_{n\ll \frac{rT^{1+\epsilon}}{eG^2}}\frac{1}{\sqrt{n}}
\sum_{n^2G^{2-\epsilon}\ll L\ll \min(nrT^{1+\epsilon}/e,n^2T^{2-\epsilon})}
\sum_{j=0}^2\Sigma_{2,2}^{(j)}(r,n,L),
\end{equation}
where
\begin{equation}\label{Sigma22j def}
\Sigma_{2,2}^{(j)}(r,n,L)=
\sum_{\substack{c|4rn/e_2\\c\equiv \pm j\Mod{4}}}
\mathop{{\sum}^*}_{a \Mod{c}}e\left(\frac{-an^2}{c}\right)
\sum_{l=-\infty}^{\infty}
\frac{\Zag_{l}(1/2)}{\sqrt{l}}g_2(l)e\left(\frac{al}{c}\right),
\end{equation}
Note that setting $n = cq$ in \eqref{g2 def} under the assumption that $e = r$ yields the function studied in \cite[(4.22)]{Frolsym2}; therefore, the analysis of the integral transforms \eqref{Voronoi MTeq0}, \eqref{phi hat+ to Phipm def0}, and \eqref{phi hat- to Phipm def0} is highly analogous. We first show that the main terms arising from the Voronoi summation formula are negligible. The following lemma serves as the  analogue of \cite[Lemma 9]{Frolsym2}.

%%%%%%%%%%%%%%%%%%%%%%%%%%%%%%%%%%%%%%%%%%%%%
\begin{lem}\label{lem:g2 VoronoiMT}
For $\alpha=1,2,4$ the following estimate holds:
\begin{equation}\label{Res(g2) est0}
\Resid(g_2;\alpha c)\ll T^{-A}.
\end{equation}
\end{lem}
\begin{proof}
It follows from \eqref{Voronoi MTeq0} and \eqref{g2 def} that
\begin{equation}\label{Res(g2) eq1}
\Resid(g_2;\alpha c)\ll \frac{L^{1/4}}{c}\int_0^{\infty}
f_2(y)\exp(-2iTA(y))
\left(\log(cL)+\log(y)\right)dy,
\end{equation}
where
\begin{equation}\label{g2 f2 def}
f_2(y)=\frac{U(y)}{y^{1/4}\sqrt{y-n^2/L}}\exp\left(-G^2A(y)^2\right)V\left(eL\frac{y-n^2/L}{4rn},T\right),
\end{equation}
\begin{equation}\label{Res(g2) Ay def}
A(y)=\arcosh\frac{y+n^2/L}{y-n^2/L}.
\end{equation}
To estimate this integral, we apply  Lemma \ref{Lemma BKY} with  (cf. \cite[(5.3),(5.4)]{Frolsym2})
\begin{equation}\label{BKYcond1}
R=Y=\frac{Tn}{\sqrt{L}},\, P=1,\, X=T^{\epsilon},\, V=T^{-2\epsilon}.
\end{equation}
Since $L\ll n^2T^{2-\epsilon}$, we prove \eqref{Res(g2) est0}.
\end{proof}
%%%%%%%%%%%%%%%%%%%%%%%%%%%%%%%%%%%%%%%%%%%%%%%
Next, we obtain the following analogue of \cite[Lemma 10]{Frolsym2}.

%%%%%%%%%%%%%%%%%%%%%%%%%%%%%%%%%%%%%%%%%%%%%%%%%%%%%%%%%%%%%%
\begin{lem}\label{lem g2hat large n}
Let $\alpha$ be some fixed positive number. The following estimate holds:
\begin{equation}\label{g2hat- est0}
\widehat{g_2}\left(-\frac{\alpha m}{c^2}\right)\ll T^{-A}.
\end{equation}
Furthermore, for $m\gg T^{\epsilon}(Tnc/L)^2$ and $m\ll T^{\epsilon}c^2/L$ we have
\begin{equation}\label{g2hat+ est0}
\widehat{g_2}\left(\frac{\alpha m}{c^2}\right)\ll T^{-A}.
\end{equation}
\end{lem}
\begin{proof}
It follows from \eqref{g2 def}, \eqref{phi hat+ to Phipm def0}, \eqref{phi hat- to Phipm def0}, \eqref{Phi++--def} and  \eqref{Phi+--+def} that
\begin{equation}\label{g2 hat- est1}
\widehat{g_2}\left(-y\right)\ll L^{1/4}\sqrt{y}
\int_0^{\infty}f_2(x)\exp(-2iTA(x))K_0(2\sqrt{xLy})dx,
\end{equation}
\begin{equation}\label{g2 hat+ est1}
\widehat{g_2}\left(y\right)\ll L^{1/4}\sqrt{y}
\int_0^{\infty}f_2(x)\exp(-2iTA(x))B_0(2\sqrt{xLy})dx,
\end{equation}
where $f_2(x)$ is defined in \eqref{g2 f2 def}, $A(x)$ is given by \eqref{Res(g2) Ay def}, and $B_0$ is either $Y_0$ or $J_0$.
Due to the exponential decay of the $K$-Bessel function, we obtain the estimate \eqref{g2hat- est0} when $m \gg c^2 T^{\epsilon} / L$. In the complementary range $m \ll c^2 T^{\epsilon} / L$, we bound both \eqref{g2 hat- est1} and \eqref{g2 hat+ est1} by applying Lemma \ref{Lemma BKY}. Note that the only structural difference between \eqref{g2 hat- est1}, \eqref{g2 hat+ est1} and \eqref{Res(g2) eq1} is the presence of the Bessel functions. Since $Ly \ll T^{\epsilon}$, it follows from \cite[(10.6.7), (10.7.3)--(10.7.5)]{HMF} and \cite[(10.29.5), (10.27.3), (10.30.2)]{HMF} that
\begin{equation}\label{g2 hat+- est1}
\frac{d^{j}}{dx^j}\left(B_0(2\sqrt{xLy})+K_0(2\sqrt{xLy})\right)\ll T^{\epsilon}.
\end{equation}
Therefore, as in the proof of Lemma \ref{lem:g2 VoronoiMT}, we apply Lemma \ref{Lemma BKY} with
\begin{equation}\label{BKYconditions 2}
R=Y=\frac{Tn}{\sqrt{L}},\, P=1,\, X=T^{\epsilon},\, V=T^{-2\epsilon},
\end{equation}
thus establishing \eqref{g2hat- est0} for all $m$, as well as \eqref{g2hat+ est0} in the range $m \ll c^2 T^{\epsilon}/L$.
It remains to establish \eqref{g2hat+ est0} in the range $m \gg T^{\epsilon}(Tnc/L)^2$. To this end, we utilize the following result (cf. \cite[Lemma~6.1]{Har}):
\begin{equation}\label{Harcos est}
\int_0^{\infty}F(x)B_0(a\sqrt{x})dx=\pm\left(\frac{2}{a}\right)^j\int_0^{\infty}F^{(j)}(x)
x^{j/2}B_j(a\sqrt{x})dx,
\end{equation}
where $F$ is a smooth compactly supported function. We show that (see \cite[Lemma 10]{Frolsym2})
\begin{multline}\label{g2 hat+ est2}
\widehat{g_2}\left(y\right)\ll \frac{L^{1/4}\sqrt{y}}{(Ly)^{j/2}}
\int_0^{\infty}\frac{d^{j}}{dx^j}\left(f_2(x)\exp(-2iTA(x))\right)x^{j/2}B_j(2\sqrt{xLy})dx\ll\\\ll
\frac{L^{1/4}\sqrt{y}}{(Ly)^{j/2+1/4}}\left(\frac{Tn}{\sqrt{L}}\right)^j\ll
y^{1/4}\left(\frac{Tn}{L\sqrt{y}}\right)^j\ll T^{-A},
\end{multline}
provided that $y\gg (T^{1+\epsilon}n/L)^2$.
\end{proof}
%%%%%%%%%%%%%%%%%%%%%%%%%%%%%%%%%%%%%%%%%%%%%%
It remains to analyze the behavior of $\widehat{g_2}\left(\alpha m/c^2\right)$ in the range $T^{\epsilon}c^2/L \ll m \ll T^{\epsilon}(Tnc/L)^2$. This analysis is carried out in the next lemma, which serves as a simplified variant of \cite[Lemma 10]{Frolsym2}.
%%%%%%%%%%%%%%%%%%%%%%%%%%%%%%%%%%%%%%%%%%%%%%%%
\begin{lem}\label{lem g2hat asympt}
There exist constants  $0<k_1<1$ and $k_2>16$ such that for $\sqrt{m}\le k_1\frac{Tnc}{L}$ and  $\sqrt{m}\ge k_2\frac{Tnc}{L}$  we have
\begin{equation*}\label{g2hat asympt00}
\widehat{g_2}\left(\frac{m}{c^2}\right)\ll T^{-A},
\end{equation*}
and for $k_1\frac{Tnc}{L}<\sqrt{m}<k_2\frac{Tnc}{L}$ we have
\begin{equation}\label{g2hat asympt0}
\widehat{g_2}\left(\frac{m}{c^2}\right)\ll\frac{T^{\epsilon}}{L^{1/4}}.
\end{equation}
\end{lem}
\begin{proof}
By virtue of the estimates established in Lemma \ref{lem g2hat large n}, we can henceforth restrict our attention to the range $T^{\epsilon}c^2/L \ll m \ll T^{\epsilon}(Tnc/L)^2$. Consequently, the argument of the Bessel functions is sufficiently large to allow for their representation via the standard asymptotic expansion \cite[Formula 8.451]{GR} (cf. \cite[(5.9)]{Frolsym2}), which yields
\begin{equation}\label{g2 hat+ est3}
\widehat{g_2}\left(y\right)=y^{1/4}\sum_{\pm}
\int_0^{\infty}f_2(x)\exp(-2iTh_{\pm}(x))W(Lxy)\frac{dx}{x^{1/4}}+O\left(\frac{y^{1/4}}{(Ly)^k}\right),
\end{equation}
where
\begin{equation}\label{g2 h def}
h_{\pm}(x)=-TA(x)\pm\sqrt{Lxy},
\end{equation}
and $W(z)$ is a smooth function such that $z^jW^{(j)}(z)\ll1$. Using \eqref{Res(g2) Ay def}, we obtain
\begin{equation}\label{hpm deriv}
h'_{\pm}(x)=\frac{TnL^{1/2}}{(Lx-n^2)x^{1/2}}\pm\frac{\sqrt{Ly}}{2\sqrt{x}}.
\end{equation}
Therefore, in the positive sign case, there is no saddle point, and
\begin{equation}\label{h+ deriv}
h'_{+}(x)\asymp\frac{Tn}{L^{1/2}}+\sqrt{Ly}\ll \frac{T^{1+\epsilon}n}{L^{1/2}},
\end{equation}
where the final inequality follows from
 $y=m/c^2\ll T^{\epsilon}(Tn/L)^2$. Accordingly,  apply Lemma \ref{Lemma BKY} with the parameter assignment
\begin{equation}\label{BKY Lem8.1 conditions}
X=T^{\epsilon},\,R=\frac{T^{1+\epsilon}n}{\sqrt{L}},\,Y=T^{\epsilon}\frac{Tn}{\sqrt{L}},\,P=1,\, V=T^{-\epsilon},
\end{equation}
which establishes that the contribution from the positive sign case is negligible. In the negative sign case, a saddle point exists, and proceeding precisely as in the proof of \cite[Lemma 10]{Frolsym2}, we obtain the estimate
\begin{equation*}
\widehat{g_2}\left(y\right)\ll\frac{T^{\epsilon}}{L^{1/4}}+\frac{L^{1/4}T^{\epsilon}}{Tn}\ll \frac{T^{\epsilon}}{L^{1/4}},
\end{equation*}
thus establishing \eqref{g2hat asympt0}.
\end{proof}
%%%%%%%%%%%%%%%%%%%%%%%%%%%%%%%%%%%%%%%
Since the expressions in \eqref{Sigma22 toSigma22j} and \eqref{Sigma22j def} possess precisely the same structure as the sums in \cite[(8.7)--(8.9)]{BF2momWA2026}, we proceed as in the proof of \cite[Lemma 8.2]{BF2momWA2026}. The crucial difference is that currently, by virtue of Lemma \ref{lem:g2 VoronoiMT}, the main term arising from the Voronoi summation formula is negligible. Therefore, applying Lemmas \ref{lem g2hat large n} and \ref{lem g2hat asympt}, we obtain the following analogue of the estimate \cite[(8.29)]{BF2momWA2026}:
\begin{multline}\label{Sigma22j est1}
\Sigma_{2,2}^{(j)}(r,n,L)\ll T^{\epsilon}
\sum_{\substack{c|4rn/e_2\\c\equiv \pm j\Mod{4}}}
\sum_{l\sim (Tnc/L)^2}
\frac{|\Zag_{l}(1/2)|}{L^{1/4}\sqrt{l}}c^{1/2+\epsilon}\sqrt{(l,n^2,c)}\ll\\\ll
\sum_{\substack{c|4rn/e_2\\c\equiv \pm j\Mod{4}}}\frac{T^{1+\epsilon}nc^{3/2+\epsilon}}{L^{5/4}}\ll
\frac{T^{1+\epsilon}n(rn/e_2)^{3/2+\epsilon}}{L^{5/4}}.
\end{multline}
Substituting \eqref{Sigma22j est1} to \eqref{Sigma22 toSigma22j}, we find that
\begin{multline}\label{Sigma22 est1}
\Sigma_{2,2}(r;h)\ll
G\sqrt{T}\sum_{e|r^2}\frac{1}{\sqrt{re_1}}\sum_{n\ll \frac{rT^{1+\epsilon}}{eG^2}}
\frac{T^{1+\epsilon}\sqrt{n}(rn/e_2)^{3/2+\epsilon}}{(nG)^{5/2-\epsilon}}\ll\\\ll
\frac{T^{2+\epsilon}}{G^{5/2}}\sum_{e|r^2}\frac{(r/e_2)^{3/2+\epsilon}}{\sqrt{ee_1}}
\ll\frac{T^{2+\epsilon}r^{3/2+\epsilon}}{G^{5/2}}.
\end{multline}
%%%%%%%%%%%%%%%%%%%%%%%%%%%%%%%%%%%
We now turn our attention to \eqref{Sigma21 eq2}. Proceeding in a manner analogous to our treatment of \eqref{Sigma22 eq2}, we obtain
\begin{equation}\label{Sigma21 toSigma21j}
\Sigma_{2,1}(r;h)=
\frac{1}{\sqrt{\pi}}\sum_{e|r^2}\frac{1}{\sqrt{re_1}}\sum_{n\ll \frac{rT^{-1+\epsilon}}{e}}
\sum_{n^2T^{2-\epsilon}\ll L\ll nrT^{1+\epsilon}/e}
\sum_{j=0}^2\Sigma_{2,1}^{(j)}(r,n,L),
\end{equation}
where
\begin{equation}\label{Sigma21j def}
\Sigma_{2,1}^{(j)}(r,n,L)=
\sum_{\substack{c|4rn/e_2\\c\equiv \pm j\Mod{4}}}
\mathop{{\sum}^*}_{a \Mod{c}}e\left(\frac{-an^2}{c}\right)
\sum_{l=-\infty}^{\infty}
\frac{\Zag_{l}(1/2)}{\sqrt{l}}g_3(l)e\left(\frac{al}{c}\right),
\end{equation}
\begin{equation}\label{g3 def}
g_3(l)=\frac{\sqrt{l}}{\sqrt{l^2-n^4}}U\left(\frac{l}{L}\right)
I\left(e\frac{l-n^2}{4rn}, 2\frac{l+n^2}{l-n^2}\right).
\end{equation}
The key structural difference between the pairs \eqref{Sigma21 toSigma21j}, \eqref{Sigma21j def} and \eqref{Sigma22 toSigma22j}, \eqref{Sigma22j def} is that the weight function $g_3(l)$ does not oscillate, unlike $g_2(l)$. Consequently, the analysis of the integral transforms of $g_3(l)$ in the Voronoi summation formula is performed in a slightly different manner, mirroring the computations in \cite[Lemma 8.1]{BF2momWA2026}.

%%%%%%%%%%%%%%%%%%%%%%%%%%%%%%%%%%%%%%%%%
\begin{lem}\label{lem:g3 est}
For $m \in \mathbb{Z}$, $L \gg n^2 T^{2-\epsilon}$, $\alpha \in \{\pi^2, 4\pi^2\}$, and any $A > 1$, it holds that
\begin{equation}\label{g3 est1}
\widehat{g_3}\left(\frac{\alpha m}{c^2}\right)\ll\frac{T^{1+\epsilon}G}{(L|m|/c^2)^{A}} \quad\hbox{if}\quad \frac{L|m|}{c^2}\gg T^{\epsilon},
\end{equation}
\begin{equation}\label{g3 est2}
\widehat{g_3}\left(\frac{\alpha m}{c^2}\right)\ll T^{1+\epsilon}G\sqrt{\frac{|m|}{c^2}}.
\end{equation}
\end{lem}
\begin{proof}
Let $y: = \alpha m / c^2$. It follows from \eqref{g3 def}, \eqref{phi hat+ to Phipm def0}, \eqref{phi hat- to Phipm def0}, \eqref{Phi++--def} and  \eqref{Phi+--+def} that
\begin{equation}\label{g3 hat est1}
\widehat{g_3}\left(y\right)\ll \sqrt{y}
\int_0^{\infty}g_3(x)B_0(2\sqrt{xy})\frac{dx}{\sqrt{x}},
\end{equation}
where $B_0$ denotes the $K_0$, $Y_0$, or $J_0$ Bessel function. In the case where $B_0 = K_0$ (which corresponds to $y < 0$), we have
\begin{equation}\label{g3 hat est2}
\widehat{g_3}\left(-y\right)\ll \sqrt{y}
\int_0^{\infty}
\frac{U\left(x\right)}{\sqrt{x^2-n^4/L^2}}
I\left(e\frac{xL-n^2}{4rn}, 2\frac{x+n^2/L}{x-n^2/L}\right)
K_0(2\sqrt{xyL})dx.
\end{equation}
Thus, owing to the exponential decay of the $K$-Bessel function in the range $yL \gg T^{\epsilon}$, combining this behavior with the estimate \eqref{I(m,x)x=2 est} yields \eqref{g3 est1}. In the complementary range $yL \ll T^{\epsilon}$, applying \cite[(10.30.3)]{HMF} and \eqref{I(m,x)x=2 est}, we obtain
\begin{equation}\label{g3 hat est3}
\widehat{g_3}\left(-y\right)\ll T^{1+\epsilon}G\sqrt{y}.
\end{equation}
Similarly using \cite[(10.7.1)]{HMF} we get \eqref{g3 est2} in the case of $y>0$.
It remains to prove \eqref{g3 est1} in the range $y > 0$. In this case, we apply the identity \eqref{Harcos est}. To simplify the analysis of the derivatives of the function $I$ defined in \eqref{g3 def}, we first transform the right-hand side of \eqref{g3 hat est1} by utilizing \eqref{integralIgeq2}:
\begin{multline}\label{g3 hat est4}
\widehat{g_3}\left(y\right)\ll \sqrt{y}
\int_{-\infty}^{\infty}\frac{th(t)}{\cosh(\pi t)}
\frac{\Gamma(1/4+it)\Gamma(3/4+it)}{\Gamma(1+2it)}\sin\left( \pi(1/4-it)\right)\\ \times
\int_0^{\infty}\frac{U\left(x\right)}{\sqrt{x^2-n^4/L^2}}
V\left(e\frac{xL-n^2}{4rn},t\right)
\left(\frac{x-n^2/L}{x+n^2/L}\right)^{2it}\\ \times
{}_2F_{1}\left(\frac{1}{4}+it,\frac{3}{4}+it,1+2it;\frac{(x-n^2/L)^2}{(x+n^2/L)^2} \right)
B_0(2\sqrt{xyL})dxdt.
\end{multline}
To estimate the higher derivatives of the functions appearing in the integrand over $x$ in \eqref{g3 hat est4}, we exploit the fact that the hypergeometric function satisfies a second-order linear differential equation. More precisely, according to \cite[Section~2.7.2, formulas~(7)--(9)]{BE}, the function

\begin{equation}\label{Y def}
Y_t(x)=x^{1/2+it}(1-x)^{1/2}{}_2F_{1}\left(\frac{1}{4}+it,\frac{3}{4}+it,1+2it;x\right)
\end{equation}
satisfies the differential equation
\begin{equation}\label{Y dif eq}
Y_t''+\left(\frac{1+4t^2}{4x^2}+\frac{1}{4(1-x)^2}+\frac{4t^2+5/4}{4x(1-x)}\right)Y_t=0.
\end{equation}
Using \eqref{Y def}, we rewrite \eqref{g3 hat est4} as
\begin{equation}\label{g3 hat est6}
\widehat{g_3}\left(y\right)\ll \sqrt{\frac{L|y|}{n^2}}
\int_{-\infty}^{\infty}\frac{th(t)}{|t|^{1/2}}\int_0^{\infty}F(x)
Y_t\left(\frac{(x-n^2/L)^2}{(x+n^2/L)^2} \right)
B_0(2\sqrt{xyL})dxdt,
\end{equation}
where
\begin{equation}\label{g3 F def}
F(x):=V\left(e\frac{xL-n^2}{4rn},t\right)\frac{U\left(x\right)(x+n^2/L)^{3/2}}{(x-n^2/L)^{3/2}\sqrt{x}}.
\end{equation}
Applying \eqref{Harcos est}, we obtain
\begin{equation}\label{g3 hat est7}
\widehat{g_3}\left(y\right)\ll \frac{\sqrt{L|y|/n^2}}{(L|y|)^{j/2}}
\int_{-\infty}^{\infty}\frac{th(t)}{|t|^{1/2}}\int_0^{\infty}\frac{d^j}{dx^j}\left(F(x)
Y_t\left(\frac{(x-n^2/L)^2}{(x+n^2/L)^2} \right)\right)
x^{j/2}B_j(2\sqrt{xyL})dxdt.
\end{equation}
To estimate the derivatives of $Y_t$, we analyze the expression within the brackets in \eqref{Y dif eq}. Note that the argument of the hypergeometric function in \eqref{g3 hat est4} can be bounded as follows:
\begin{equation}\label{Y arg}
0<1-\frac{(x-n^2/L)^2}{(x+n^2/L)^2}=\frac{4xn^2/L}{(x+n^2/L)^2}\ll\frac{n^2}{L}\ll T^{\epsilon-2}.
\end{equation}
Consequently, the expression within the brackets in \eqref{Y dif eq} is bounded by $T^{\epsilon} / (1-x)^2$. Given that \eqref{g3 hat est7} coincides with \cite[(8.21)]{BF2momWA2026} and \eqref{Y dif eq} shares a highly similar structure with \cite[(8.19)]{BF2momWA2026}, the remainder of the proof of \eqref{g3 est1} proceeds identically to that of \cite[Lemma 8.1]{BF2momWA2026}.
\end{proof}
%%%%%%%%%%%%%%%%%%%%%%%%%%%%%%%%%%%%%%%%
Proceeding as in \cite[Lemmas 8.2--8.4]{BF2momWA2026} and utilizing Lemma \ref{lem:g3 est}, we obtain
\begin{multline}\label{Sigma21 asympt1}
\Sigma_{2,1}(r;h)=\frac{1}{2\sqrt{\pi}}
\sum_{e|r^2}\frac{e_2}{\sqrt{re}}\sum_{n\ll rT^{-1+\epsilon}e^{-1}}
\sum_{q|rn/e_2}\frac{\ups_q(n^2)}{\sqrt{q}}
\sum_{n^2T^{2-\epsilon}\ll L\ll nrT^{1+\epsilon}/e}\int_0^{\infty}g_3(y)\\\times
\left(\log\frac{2y}{\pi (4q)^2}-\frac{\pi}{2}+3\gamma\right)\frac{dy}{\sqrt{y}}+\\
+O\Bigl(\sum_{e|r^2}\frac{1}{\sqrt{re_1}}\sum_{n\ll \frac{rT^{-1+\epsilon}}{e}}
\sum_{n^2T^{2-\epsilon}\ll L\ll nrT^{1+\epsilon}/e}\sum_{c|4rn/e_2}
\sum_{|l|\ll c^2T^{\epsilon}/L}\frac{|\Zag_{l}(1/2)|}{\sqrt{l}}\\\times
T^{1+\epsilon}G\sqrt{\frac{|l|}{c^2}}c^{1/2+\epsilon}\sqrt{(l,n^2,c)}
\Bigr).
\end{multline}
Estimating the error term reveals that it is bounded by
\begin{multline}\label{Sigma21 asympt2}
\sum_{e|r^2}\frac{T^{1+\epsilon}G}{\sqrt{re_1}}\sum_{n\ll \frac{rT^{-1+\epsilon}}{e}}
\sum_{n^2T^{2-\epsilon}\ll L\ll nrT^{1+\epsilon}/e}\sum_{c|4rn/e_2}\frac{c^{3/2+\epsilon}}{L}\ll\\\ll
\sum_{e|r^2}\frac{T^{1+\epsilon}G}{\sqrt{re_1}}\sum_{n\ll \frac{rT^{-1+\epsilon}}{e}}
\frac{(rn/e_2)^{3/2+\epsilon}}{n^2T^2}\ll
\sum_{e|r^2}\frac{T^{1+\epsilon}G}{\sqrt{re_1}}\frac{(r/e_2)^{3/2+\epsilon}(r/e)^{1/2}}{T^{5/2}}\ll\frac{T^{1+\epsilon}Gr^{3/2}}{T^{5/2}}, 
\end{multline}
and therefore,
\begin{multline}\label{Sigma21 asympt3}
\Sigma_{2,1}(r;h)=\frac{1}{2\sqrt{\pi}}
\sum_{e|r^2}\frac{e_2}{\sqrt{re}}\sum_{n\ll rT^{-1+\epsilon}e^{-1}}
\sum_{q|rn/e_2}\frac{\ups_q(n^2)}{\sqrt{q}}\\\times
\sum_{n^2T^{2-\epsilon}\ll L\ll nrT^{1+\epsilon}/e}\int_0^{\infty}g_3(y)
\left(\log\frac{y}{8\pi q^2}-\frac{\pi}{2}+3\gamma\right)\frac{dy}{\sqrt{y}}+
+O\left(\frac{T^{1+\epsilon}Gr^{3/2}}{T^{5/2}}\right).
\end{multline}
Since $G \ll T^{1-\epsilon}$, we deduce from \eqref{SIN=Sigma21+Sigma22}, \eqref{Sigma22 est1}, and \eqref{Sigma21 asympt3} that
\begin{multline}\label{SIN asympt}
\SIN(r;h)=\frac{1}{2\sqrt{\pi}}
\sum_{e|r^2}\frac{e_2}{\sqrt{re}}\sum_{n\ll rT^{-1+\epsilon}e^{-1}}
\sum_{q|rn/e_2}\frac{\ups_q(n^2)}{\sqrt{q}}\\\times
\sum_{n^2T^{2-\epsilon}\ll L\ll nrT^{1+\epsilon}/e}\int_0^{\infty}g_3(y)
\left(\log\frac{y}{8\pi q^2}-\frac{\pi}{2}+3\gamma\right)\frac{dy}{\sqrt{y}}+
+O\left(\frac{T^{2+\epsilon}r^{3/2+\epsilon}}{G^{5/2}}\right).
\end{multline}
To simplify the evaluation of the main term, we proceed as in \cite[Section 10]{BF2momWA2026} and replace $\log y$ by $\log(y-n^2)$. By \eqref{I(m,x)x=2 est} and \cite[(10.4)]{BF2momWA2026}, the error introduced by this substitution is bounded by
\begin{equation}\label{error log change}
\sum_{e|r^2}\frac{e_2}{\sqrt{re}}\sum_{n\ll rT^{-1+\epsilon}e^{-1}}
\sum_{q|rn/e_2}\frac{|\ups_q(n^2)|}{\sqrt{q}}
\frac{T^{1+\epsilon}G}{T^2}\ll \frac{T^{\epsilon}Gr^{1/2}}{T^2}\ll \frac{T^{2+\epsilon}r^{3/2+\epsilon}}{G^{5/2}}.
\end{equation}
Next, we extend the summations over $n$ and $L$ in \eqref{SIN asympt} to the complete ranges. First, we complete the sum over $n$, which is straightforward since the sum over $L$ becomes empty for $n \gg r T^{-1+\epsilon} e^{-1}$. We then enlarge the summation over $L$ to the range $L \gg n^2 T^{2-\epsilon}$ by leveraging the rapid decay of \eqref{Vestimate}. Finally, for $L \ll n^2 G^{2-\epsilon}$, the main term is negligible by virtue of \eqref{2mom SIN I eq1}, and for $n^2 G^{2-\epsilon} \ll L \ll n^2 T^{2-\epsilon}$, it is negligible by virtue of Lemma \ref{lem:g2 VoronoiMT}.
As a result,
\begin{equation}\label{SIN asympt2}
\SIN(r;h)=\frac{1}{2\sqrt{\pi}}
\sum_{e|r^2}\frac{e_2}{\sqrt{re}}\sum_{n=1}^{\infty}
\sum_{q|rn/e_2}\frac{\ups_q(n^2)}{\sqrt{q}}I_{>}(r,n,e)
+O\left(\frac{T^{2+\epsilon}r^{3/2+\epsilon}}{G^{5/2}}\right),
\end{equation}
where (cf. \eqref{g3 def})
\begin{equation}\label{I>def}
I_{>}(r,n,e)=
\int_{n^2}^{\infty}\left(\log\frac{y-n^2}{8\pi q^2}-\frac{\pi}{2}+3\gamma\right)
I\left(e\frac{y-n^2}{4rn}, 2\frac{y+n^2}{y-n^2}\right)\frac{dy}{\sqrt{y^2-n^4}}.
\end{equation}
Note that \eqref{I>def} is an analogue of \cite[(10.7)]{BF2momWA2026}.   Using \eqref{I(m,x) def} and \eqref{eq:integralI},  we show that
\begin{multline}\label{I>def2}
I_{>}(r,n,e)=\frac{2i}{\pi}\int_{-\infty}^{\infty}\frac{th(t)}{\cosh(\pi t)}V\left(e\frac{y-n^2}{4rn},t\right)
\int_{n^2}^{\infty}\left(\log\frac{y-n^2}{8\pi q^2}-\frac{\pi}{2}+3\gamma\right)\\\times
I_t\left(2\frac{y+n^2}{y-n^2}\right)\frac{dy}{\sqrt{y^2-n^4}}dt.
\end{multline}
Applying the change of variables $y = n^2 x$ and utilizing \eqref{approx.fun.eq.Vdef}, we obtain
\begin{equation}\label{I>(r,n,e) to I>(t,z)}
I_{>}(r,n,e)=\frac{2i}{\pi}\int_{-\infty}^{\infty}\frac{th(t)}{\cosh(\pi t)}
\frac{1}{2\pi i}\int_{(a)}\frac{L_{\infty}(1/2+z,t)}{L_{\infty}(1/2,t)}\zeta(1+2z)G(z)
\left(\frac{en}{4r}\right)^{-z}
I_{>}(t,z)\frac{dz}{z},
\end{equation}
\begin{equation}\label{I>(t,z)def}
I_{>}(t,z)=
\int_{1}^{\infty}\left(\log\frac{(x-1)n^2}{8\pi q^2}-\frac{\pi}{2}+3\gamma\right)
I_t\left(2\frac{x+1}{x-1}\right)\frac{dx}{(x-1)^{1/2+z}(x+1)^{1/2}}.
\end{equation}
Let us introduce the function
\begin{equation}\label{I>0(t,z)def}
I_{>,0}(t,z):=\int_{1}^{\infty}I_t\left(2\frac{x+1}{x-1}\right)\frac{dx}{(x-1)^{1/2+z}(x+1)^{1/2}},
\end{equation}
so that
\begin{equation}\label{I>(t,z) to I>0(t,z)}
I_{>}(t,z)=\left(\log\frac{n^2}{8\pi q^2}-\frac{\pi}{2}+3\gamma-\frac{\partial}{\partial z}\right)I_{>,0}(t,z).
\end{equation}
%%%%%%%%%%%%%%%%%%%%
\begin{rem}\label{rem:it=k-1/2}
Note that \eqref{I_t(x)def} shares a close structural similarity with \cite[(10.8)]{BF2momWA2026}. Specifically, substituting $k-1/2$ for $it$ in \eqref{I_t(x)def} yields \cite[(10.8)]{BF2momWA2026}. Consequently, if we substitute $k-1/2$ for $it$ in \eqref{I>0(t,z)def} and multiply the resulting expression by $(-1)^k \sqrt{\pi}$, we obtain \cite[(10.14)]{BF2momWA2026}.
\end{rem}
%%%%%%%%%%%%%%%%
This observation allows us to establish the following result:
\begin{equation}\label{I>0(t,z) eq0}
I_{>,0}(t,z)=\frac{\sqrt{2}\Gamma^2(z)\Gamma(1/2-z+2it)}{\sqrt{\pi}\Gamma(1/2+z+2it)}\sin(\pi(1/4-it)),\,0<\Re{z}<\Re(1/2+2it),
\end{equation}
which can be proved directly from \eqref{I>0(t,z)def} and \eqref{I_t(x)def} in an identical manner to the proof of \cite[Lemma 10.1]{BF2momWA2026}. For subsequent computations, it is convenient to rewrite \eqref{I>(r,n,e) to I>(t,z)} as
\begin{multline}\label{I>(r,n,e) to I>(t,z)v2}
I_{>}(r,n,e)=\frac{2i}{\pi}\int_{0}^{\infty}\frac{th(t)}{\cosh(\pi t)}
\frac{1}{2\pi i}\int_{(a)}\frac{L_{\infty}(1/2+z,t)}{L_{\infty}(1/2,t)}\zeta(1+2z)G(z)\\\times
\left(\frac{en}{4r}\right)^{-z}
\left(I_{>}(t,z)-I_{>}(-t,z)\right)\frac{dz}{z}.
\end{multline}
It follows from  \eqref{I>0(t,z) eq0} and \cite[(5.11.13)]{HMF} that for $t>0$ 
\begin{multline}\label{I>0(t,z)-I>0(-t,z)}
I_{>,0}(t,z)-I_{>,0}(-t,z)=\\=\sqrt{\frac{2}{\pi}}\Gamma^2(z)\left(\frac{\sin(\pi(1/4-it))}{(1/2+2it)^{2z}}-
\frac{\sin(\pi(1/4+it))}{(1/2-2it)^{2z}}\right)\left(1+O\left(\frac{1+|z|}{1+|t|}\right)\right)=\\=
\sqrt{\frac{2}{\pi}}\frac{\Gamma^2(z)}{(2t)^{2z}}
\left(\sin(\pi(1/4-it))e^{-\pi iz}-\sin(\pi(1/4+it))e^{\pi iz}\right)\left(1+O\left(\frac{1+|z|}{1+|t|}\right)\right)=\\=
\sqrt{\frac{2}{\pi}}\frac{\Gamma^2(z)}{2i(2t)^{2z}}
\left(\sqrt{2}(\cos(\pi z)+\sin(\pi z))e^{\pi t}+\sqrt{2}(\sin(\pi z)-\cos(\pi z))e^{-\pi t}\right)\\\times\left(1+O\left(\frac{1+|z|}{1+|t|}\right)\right)=
\frac{\Gamma^2(z)e^{\pi t}}{i\sqrt{\pi}(2t)^{2z}}\left(\cos(\pi z)+\sin(\pi z)\right)\left(1+O\left(\frac{1+|z|}{1+|t|}\right)\right).
\end{multline}

%%%%%%%%%%%%%%%%%%%%%%%%%%%%%%%%%%%%%%%%%%%%%%%%%%%%%%%%%%%%%%%%%%%%%%%%%%%%%%%%%%%%%%%%%%%
%%%%%%%%%%%%%%%%%%%%%%%%%%%%%%%%%%%%%%%%%%%%%%%%%%%%%%%%%%%%%%%%%%%%%%%%%%%%%%%%%%%%%%%%%%%
\section{Evaluation of $\SZE(r;h)$}\label{sec:SZE}
According to \cite[Lemma 4.3]{BF2mom}, the contribution of the summands satisfying $\frac{ne_2}{mr} \gg T^{-2+\epsilon}$ is negligible. Therefore, we immediately obtain the following analogue of \eqref{Sigma21 eq1}:
\begin{equation}\label{SZE=Sigma11+}
\SZE(r;h)=\Sigma_{1,1}(r;h)+O(T^{-A}),
\end{equation}
\begin{equation}\label{Sigma11 eq1}
\Sigma_{1,1}(r;h)=\frac{4}{\sqrt{\pi}}\sum_{e|r^2}\frac{\sqrt{r}}{\sqrt{e}}\sum_{n\ll \frac{rT^{-1+\epsilon}}{e}}
\sum_{\substack{nT^{2-\epsilon}\ll q\ll rT^{1+\epsilon}/e\\q\equiv -n\Mod{4r/e_2}}}
\frac{\Zag_{-qn}(1/2)}{\sqrt{q^2-n^2}}I\left(e\frac{q+n}{4r}, 2\frac{q-n}{q+n}\right).
\end{equation}
Proceeding in a manner analogous to Section \ref{sec:SIN} (or following the derivation of \cite[(8.10)]{BF2momWA2026}), we obtain (cf. \eqref{Sigma21 toSigma21j}--\eqref{g3 def})
\begin{equation}\label{Sigma11 toSigma11j}
\Sigma_{1,1}(r;h)=
\frac{1}{\sqrt{\pi}}\sum_{e|r^2}\frac{1}{\sqrt{re_1}}\sum_{n\ll \frac{rT^{-1+\epsilon}}{e}}
\sum_{n^2T^{2-\epsilon}\ll L\ll nrT^{1+\epsilon}/e}
\sum_{j=0}^2\Sigma_{1,1}^{(j)}(r,n,L),
\end{equation}
where
\begin{equation}\label{Sigma11j def}
\Sigma_{1,1}^{(j)}(r,n,L)=
\sum_{\substack{c|4rn/e_2\\c\equiv \pm j\Mod{4}}}
\mathop{{\sum}^*}_{a \Mod{c}}e\left(\frac{-an^2)}{c}\right)
\sum_{l=-\infty}^{\infty}
\frac{\Zag_{l}(1/2)}{\sqrt{|l|}}g_1(l)e\left(\frac{al}{c}\right),
\end{equation}
\begin{equation}\label{g1 def}
g_1(l)=\frac{\sqrt{|l|}}{\sqrt{l^2-n^4}}U\left(\frac{-l}{L}\right)
I\left(e\frac{-l+n^2}{4rn}, 2\frac{-l-n^2}{-l+n^2}\right).
\end{equation}
We now establish an analogue of Lemma \ref{lem:g3 est}.
%%%%%%%%%%%%%%%%%%%
\begin{lem}\label{lem:g1 est}
For $m \in \mathbb{Z}$, $L \gg n^2 T^{2-\epsilon}$, $\alpha \in \{\pi^2, 4\pi^2\}$, and any $A > 1$, it holds that
\begin{equation}\label{g1 est1}
\widehat{g_1}\left(\frac{\alpha m}{c^2}\right)\ll\frac{T^{1+\epsilon}G}{(L|m|/c^2)^{A}} \quad\hbox{if}\quad \frac{L|m|}{c^2}\gg T^{\epsilon},
\end{equation}
\begin{equation}\label{g1 est2}
\widehat{g_1}\left(\frac{\alpha m}{c^2}\right)\ll T^{1+\epsilon}G\sqrt{\frac{|m|}{c^2}}.
\end{equation}
\end{lem}
\begin{proof}
We first establish an estimate analogous to \eqref{I(m,x)x=2 est} in the range $0 < 2-x \ll t^{-2+\epsilon}$. From \eqref{integralIleq2} and \eqref{I(m,x) def}, it follows that
\begin{equation}\label{I(m,x)<2 est1}
I\left(me_1e_2,x\right)\ll \int_{0}^{\infty}\frac{t^{1/2}h(t)}{\cosh(\pi t)}x^{1/2}
\Bigl|{}_2F_{1}\left(1/4+it,1/4-it,1/2;\frac{x^2}{4} \right)V(me_1e_2,t)\Bigr|dt.
\end{equation}
Let $a:=\arccos(2x-1).$ It follows from \cite[Lemmas 4.1--4.2]{BF2mom} that for $t > 0$, the leading term in the asymptotic formula for the hypergeometric function in \eqref{integralIleq2} is given by
\begin{multline}\label{2F1(1/4,1/4,1/2) asympt1}
{}_2F_{1}\left(1/4+it,1/4-it,1/2;x \right)\sim\frac{e^{\pi i/2}}{\sqrt{2\pi}}\frac{\Gamma(3/4+it)}{\Gamma(1/4+it)}\frac{\sqrt{\arccos(2x-1)}}{(1-x)^{1/4}}\\\times
\left(e^{-\pi i/4-\pi t}K_0(-at)-e^{\pi i/4+\pi t}K_0(at)\right)\sim
\frac{e^{\pi i/2}}{\sqrt{2\pi}}\frac{\Gamma(3/4+it)}{\Gamma(1/4+it)}\frac{\sqrt{\arccos(2x-1)}}{(1-x)^{1/4}}\\\times
\left(-\pi e^{\pi i/4-\pi t}I_0(at)-e^{\pi i/4+\pi t}K_0(at)\right).
\end{multline}
Since \eqref{I(m,x)<2 est1} contains an additional factor of $e^{-\pi t}$, the contribution from the component of \eqref{2F1(1/4,1/4,1/2) asympt1} involving $I_0(at)$ is negligible. Therefore, we obtain
\begin{multline}\label{2F1(1/4,1/4,1/2) asympt2}
{}_2F_{1}\left(1/4+it,1/4-it,1/2;x\right)\sim
\frac{e^{-\pi i/4+\pi t}}{\sqrt{2\pi}}\frac{\Gamma(3/4+it)}{\Gamma(1/4+it)}\frac{\sqrt{\arccos(2x-1)}}{(1-x)^{1/4}}\\\times
K_0(t\arccos(2x-1)).
\end{multline}
Substituting \eqref{2F1(1/4,1/4,1/2) asympt2} into \eqref{I(m,x)<2 est1}, we prove for $0<2-x\ll t^{-2+\epsilon}$ the estimate (cf. \cite[(10.30.3)]{HMF})
\begin{equation}\label{I(m,x)<2 est2}
I\left(me_1e_2,x\right)\ll \int_{0}^{\infty}th(t)
\Bigl|K_0(t\arccos(x^2/2-1))V(me_1e_2,t)\Bigr|dt\ll T^{1+\epsilon}G|\log(2-x)|.
\end{equation}
By applying this estimate, we establish \eqref{g1 est2} and \eqref{g1 est1} in the range $m < 0$ in an identical manner to the proof of Lemma \ref{lem:g3 est}. To establish \eqref{g1 est1} for $m > 0$, we proceed analogously (cf. \eqref{g1 def}, \eqref{phi hat+ to Phipm def0}, \eqref{Phi++--def}, \eqref{I(m,x) def}, and \eqref{integralIleq2}), first obtaining
\begin{multline}\label{g1 hat est1}
\widehat{g_1}\left(y\right)\ll \sqrt{y}
\int_{0}^{\infty}\frac{th(t)}{\cosh(\pi t)}
\Gamma(1/4+it)\Gamma(1/4-it)\cos\left( \pi(1/4-it)\right)
\int_0^{\infty}\frac{U\left(x\right)}{\sqrt{x^2-n^4/L^2}}\\ \times
V\left(e\frac{xL+n^2}{4rn},t\right)
\left(\frac{x-n^2/L}{x+n^2/L}\right)^{1/2}{}_2F_{1}\left(\frac{1}{4}+it,\frac{1}{4}-it,\frac{1}{2};\frac{(x-n^2/L)^2}{(x+n^2/L)^2} \right)
B_0(2\sqrt{xyL})dxdt.
\end{multline}
Then, using \cite[Section~2.7.2, formulas~(7)-(9)]{BE}, we show  that the function
\begin{equation}\label{Z def}
Z_t(x)=x^{1/4}(1-x)^{1/2}{}_2F_{1}\left(\frac{1}{4}+it,\frac{1}{4}-it,\frac{1}{2};x\right)
\end{equation}
satisfies the differential equation
\begin{equation}\label{Z dif eq}
Z_t''+\left(\frac{3}{16x^2}+\frac{1}{4(1-x)^2}+\frac{3/4-4t^2}{4x(1-x)}\right)Z_t=0.
\end{equation}
Applying \eqref{Z def}, we rewrite \eqref{g1 hat est1} as
\begin{equation}\label{g1 hat est2}
\widehat{g_1}\left(y\right)\ll \sqrt{\frac{Ly}{n^2}}
\int_{0}^{\infty}\frac{t^{1/2}h(t)}{\cosh(\pi t)}\int_0^{\infty}F(x)
Z_t\left(\frac{(x-n^2/L)^2}{(x+n^2/L)^2} \right)
B_0(2\sqrt{xyL})dxdt,
\end{equation}
where
\begin{equation}\label{g1 F def}
F(x):=V\left(e\frac{xL+n^2}{4rn},t\right)\frac{U\left(x\right)(x+n^2/L)^{1/2}}{(x-n^2/L)^{1/2}\sqrt{x}}.
\end{equation}
Then, according to \eqref{Harcos est}, we have
\begin{equation}\label{g1 hat est3}
\widehat{g_1}\left(y\right)\ll \frac{\sqrt{Ly/n^2}}{(Ly)^{j/2}}
\int_{0}^{\infty}\frac{t^{1/2}h(t)}{\cosh(\pi t)}\int_0^{\infty}\frac{d^j}{dx^j}\left(F(x)
Z_t\left(\frac{(x-n^2/L)^2}{(x+n^2/L)^2} \right)\right)
x^{j/2}B_j(2\sqrt{xyL})dxdt.
\end{equation}
Once again, to estimate the derivatives of $Z_t$, we analyze the expression within the brackets in \eqref{Z dif eq}. Since the argument of the hypergeometric function in \eqref{g1 hat est3} can be bounded as in \eqref{Y arg}, the term within the brackets in \eqref{Z dif eq} is bounded by $T^{\epsilon} / (1-x)^2$. Therefore, the remainder of the proof proceeds identically to that of \cite[Lemma 8.1]{BF2momWA2026}.

\end{proof}
%%%%%%%%%%%%%%%%%%%%%%%%%%%%%%%%%%
Since the estimates in Lemma \ref{lem:g1 est} coincide with those in Lemma \ref{lem:g3 est}, we obtain the following analogue of \eqref{Sigma21 asympt3}:
\begin{multline}\label{Sigma11 asympt1}
\Sigma_{1,1}(r;h)=\frac{1}{2\sqrt{\pi}}
\sum_{e|r^2}\frac{e_2}{\sqrt{re}}\sum_{n\ll rT^{-1+\epsilon}e^{-1}}
\sum_{q|rn/e_2}\frac{\ups_q(n^2)}{\sqrt{q}}\\\times
\sum_{n^2T^{2-\epsilon}\ll L\ll nrT^{1+\epsilon}/e}\int_0^{\infty}g_1(-y)
\left(\log\frac{y}{8\pi q^2}+\frac{\pi}{2}+3\gamma\right)\frac{dy}{\sqrt{y}}+
+O\left(\frac{T^{1+\epsilon}Gr^{3/2}}{T^{5/2}}\right).
\end{multline}
Proceeding in a manner analogous to Section \ref{sec:SIN}, we obtain the following analogue of \eqref{SIN asympt2} (cf. \eqref{SZE=Sigma11+}):
\begin{equation}\label{SZE asympt2}
\SZE(r;h)=\frac{1}{2\sqrt{\pi}}
\sum_{e|r^2}\frac{e_2}{\sqrt{re}}\sum_{n=1}^{\infty}
\sum_{q|rn/e_2}\frac{\ups_q(n^2)}{\sqrt{q}}I_{<}(r,n,e)
+O\left(\frac{T^{1+\epsilon}Gr^{3/2}}{T^{5/2}}\right),
\end{equation}
where (recall \eqref{g1 def})
\begin{equation}\label{I<def}
I_{<}(r,n,e)=
\int_{n^2}^{\infty}\left(\log\frac{y+n^2}{8\pi q^2}+\frac{\pi}{2}+3\gamma\right)
I\left(e\frac{y+n^2}{4rn}, 2\frac{y-n^2}{y+n^2}\right)\frac{dy}{\sqrt{y^2-n^4}}.
\end{equation}
Similarly to the proof of \eqref{I>(r,n,e) to I>(t,z)}, we show that
\begin{equation}\label{I<(r,n,e) to I<(t,z)}
I_{<}(r,n,e)=\frac{2i}{\pi}\int_{-\infty}^{\infty}\frac{th(t)}{\cosh(\pi t)}
\frac{1}{2\pi i}\int_{(a)}\frac{L_{\infty}(1/2+z,t)}{L_{\infty}(1/2,t)}\zeta(1+2z)G(z)
\left(\frac{en}{4r}\right)^{-z}
I_{<}(t,z)\frac{dz}{z},
\end{equation}
where
\begin{equation}\label{I<(t,z) to I<0(t,z)}
I_{<}(t,z)=\left(\log\frac{n^2}{8\pi q^2}+\frac{\pi}{2}+3\gamma-\frac{\partial}{\partial z}\right)I_{<,0}(t,z),
\end{equation}
and
\begin{equation}\label{I<0(t,z)def}
I_{<,0}(t,z):=\int_{1}^{\infty}I_t\left(2\frac{x-1}{x+1}\right)\frac{dx}{(x-1)^{1/2}(x+1)^{1/2+z}}.
\end{equation}
Using \eqref{I<0(t,z)def} and \eqref{I_t(x)def}, and proceeding in an identical manner to the proof of \cite[Lemma 10.2]{BF2momWA2026}, we obtain the following analogue of \cite[(10.31)]{BF2momWA2026} for $0 < \Re(z) < \Re(1/2 + 2it)$ (cf. also Remark \ref{rem:it=k-1/2}):
\begin{multline}\label{I<0(t,z) eq0}
I_{<,0}(t,z)=\sqrt{\frac{2}{\pi}}
\frac{\Gamma^2(z)\Gamma(1/2-z+2it)}{\Gamma(1/2+z+2it)}\frac{\sin(\pi(1/2+z-2it))\sin(\pi(1/4-it))}{\sin(\pi(1/2+2it))}+\\+
\frac{\pi 2^{-3/2-z}(1-z)}{\Gamma(7/4+it)\Gamma(7/4-it)\sin(\pi(3/4-it))}
\GenHyG{3}{2}{1-z/2,3/2-z/2,1}{7/4+it,7/4-it}{1}.
\end{multline}
Finally, we wish to obtain an asymptotic formula for $I_{<,0}(t,z)$. To this end, we first note (cf. \eqref{integralIleq2} and \eqref{eq:integralI}) that for $0 < x < 2$
\begin{equation}\label{It(x)x<2}
I_t(x)=\frac{x^{1/2}}{\pi^{1/2}}
\frac{\Gamma(1/4+it)\Gamma(1/4-it)}{\Gamma(1/2)}\cos\left( \pi(1/4+it)\right)
{}_2F_{1}\left(1/4+it,1/4-it,1/2;\frac{x^2}{4} \right).
\end{equation}
Substituting \eqref{It(x)x<2} into \eqref{I<0(t,z)def}, we have
\begin{multline}\label{I<0(t,z)eq1}
I_{<,0}(t,z)=\frac{2^{1/2}}{\pi}\Gamma(1/4+it)\Gamma(1/4-it)\cos\left( \pi(1/4+it)\right)\\\times
\int_{1}^{\infty}
{}_2F_{1}\left(1/4+it,1/4-it,1/2;\frac{(x-1)^2}{(x+1)^2} \right)
\frac{dx}{(x+1)^{1+z}}.
\end{multline}
Applying the change of variables $x = \cot^2 y$, we obtain
\begin{multline}\label{I<0(t,z)eq2}
I_{<,0}(t,z)=\frac{2^{3/2}}{\pi}\Gamma(1/4+it)\Gamma(1/4-it)\cos\left( \pi(1/4+it)\right)\\\times
\int_{0}^{\pi/4}
{}_2F_{1}\left(1/4+it,1/4-it,1/2;\cos^2(2y)\right)
\frac{\cos(y)dy}{\sin^{1-2z}(y)}.
\end{multline}
Suppose that $t>0$. Using \eqref{2F1(1/4,1/4,1/2) asympt2}, we infer that
\begin{multline}\label{I<0(t,z)eq3}
I_{<,0}(t,z)=\frac{2^{3/2}}{\pi^{3/2}}\Gamma(1/4-it)\Gamma(3/4+it)\cos\left( \pi(1/4+it)\right)e^{-\pi i/4+\pi t}
\\\times
\int_{0}^{\pi/4}
K_0(4ty)\frac{\sqrt{y}\cos^{1/2}(y)dy}{\sin^{3/2-2z}(y)}\left(1+O(t^{-1})\right).
\end{multline}
Due to the exponential decay of the $K$-Bessel function, we can first truncate the integral over $y$ at the point $t^{-1+\epsilon}$. Following this truncation, we approximate $\cos y \sim 1$ and $\sin y \sim y$, which yields
\begin{equation}\label{I<0(t,z)eq4}
I_{<,0}(t,z)=\frac{2^{3/2}}{\pi^{1/2}}\frac{\cos\left( \pi(1/4+it)\right)}{\sin\left( \pi(3/4+it)\right)}e^{-\pi i/4+\pi t}
\int_{0}^{\pi/4}
K_0(4ty)y^{2z-1}dy\left(1+O\left(\frac{1+|z|}{1+t}\right)\right).
\end{equation}
Using \cite[(6.561.16)]{GR}, we obtain
\begin{equation}\label{I<0(t,z)eq5}
I_{<,0}(t,z)=\frac{\Gamma^2(z)}{\sqrt{2\pi}(2t)^{2z}}\frac{\cos\left( \pi(1/4+it)\right)}{\sin\left( \pi(3/4+it)\right)}e^{-\pi i/4+\pi t}
\left(1+O\left(\frac{1+|z|}{1+t}\right)\right).
\end{equation}
Since the right-hand side of \eqref{I<0(t,z)eq1} becomes even after removing the factor $\cos\pi(1/4+it)$, we have
\begin{multline}\label{I<0(t,z)eq6}
I_{<,0}(t,z)-I_{<,0}(-t,z)=\\=
\frac{\Gamma^2(z)}{\sqrt{2\pi}(2t)^{2z}}\frac{\cos\left( \pi(1/4+it)\right)-\cos\left( \pi(1/4-it)\right)}{\sin\left( \pi(3/4+it)\right)}e^{-\pi i/4+\pi t}
\left(1+O\left(\frac{1+|z|}{1+t}\right)\right).
\end{multline}
Then the equality
\begin{equation*}
\frac{e^{-\pi i/4}}{\sqrt{2\pi}}\frac{\cos\left( \pi(1/4+it)\right)-\cos\left( \pi(1/4-it)\right)}{\sin\left( \pi(3/4+it)\right)}=
\frac{e^{-\pi i/4}}{\sqrt{2\pi}}(1-i)(1+O(e^{-2\pi t}))=\frac{-i}{\sqrt{\pi}}(1+O(e^{-2\pi t}))
\end{equation*}
shows that
\begin{equation}\label{I<0(t,z)eq7}
I_{<,0}(t,z)-I_{<,0}(-t,z)=\frac{-i\Gamma^2(z)}{\sqrt{\pi}(2t)^{2z}}e^{\pi t}
\left(1+O\left(\frac{1+|z|}{1+t}\right)\right).
\end{equation}

%%%%%%%%%%%%%%%%%%%%%%%%%%%%%%%%%%%%%%%%%%%%%%%%%%%%%%%%%%%%%%%%%%%%%%%%%%%%%%%%%%%%%%%%%%%
\section{Proof of Theorems \ref{thm:2mom average}  and \ref{thm:nonvan}}\label{sec:Thm 2mom}
It follows from \eqref{M2 to M1 eq3}, \eqref{2mom MTn est}, \eqref{2mom MT2 eq4}, \eqref{SIN asympt2}, and \eqref{SZE asympt2} that
\begin{multline}\label{2mom AF1}
\M_2(r;h(\cdot))=\MT_2(r;h)+\SZE(r;h)+\SIN(r;h)+\\+O\left(T^{\epsilon}\max(G,T^{2/3})+\frac{T^{2+\epsilon}r^{3/2+\epsilon}}{G^{5/2}}\right),
\end{multline}
where $\MT_2(r;h)$ is given by \eqref{2mom MT2 eq4}, and
\begin{equation}\label{SIN+SZE asympt1}
\SZE(r;h)+\SIN(r;h)=\frac{1}{2\sqrt{\pi}}
\sum_{e|r^2}\frac{e_2}{\sqrt{re}}\sum_{n=1}^{\infty}
\sum_{q|rn/e_2}\frac{\ups_q(n^2)}{\sqrt{q}}\left(I_{>}(r,n,e)+I_{<}(r,n,e)\right).
\end{equation}
It remains to verify that the first three summands on the right-hand side of \eqref{2mom AF1} are equal to the main term in \eqref{2momAF0}. To this end, we must first evaluate \eqref{SIN+SZE asympt1}. We shift the line of $t$-integration in \eqref{I>(r,n,e) to I>(t,z)} and \eqref{I<(r,n,e) to I<(t,z)} to $\Im(t) = -1/4-\delta$. This enables us to move the line of $z$-integration to $\Re(z) = 1+\epsilon$ (guaranteeing the absolute convergence of the sums over $n$ and $q$) and allows the application of \eqref{I>0(t,z) eq0} and \eqref{I<0(t,z) eq0}. After evaluating the sums over $n$ and $q$, we return the line of $t$-integration to $\Im(t) = 0$. Note that substituting \eqref{I>(r,n,e) to I>(t,z)} and \eqref{I<(r,n,e) to I<(t,z)} into \eqref{SIN+SZE asympt1}, combined with \eqref{I>(t,z) to I>0(t,z)} and \eqref{I<(t,z) to I<0(t,z)}, yields a sum over $n$ and $q$ identical to that in \cite[(10.33)]{BF2momWA2026}. Therefore, applying \cite[(10.54)]{BF2momWA2026}, we obtain the following analogue of \cite[(10.55)]{BF2momWA2026}:
\begin{multline}\label{SIN+SZE asympt2}
\SZE(r;h)+\SIN(r;h)=\frac{1}{2\sqrt{\pi}}
\sum_{e|r^2}\frac{e_2}{\sqrt{re}}
\sum_{d|\frac{r}{e_2}}\sum_{m|\frac{r}{de_2}}\frac{\mu(m)}{dm}\frac{2i}{\pi}\int_{-\infty}^{\infty}\frac{th(t)}{\cosh(\pi t)}\\\times
\frac{1}{2\pi i}\int_{(a)}\frac{L_{\infty}(1/2+z,t)}{L_{\infty}(1/2,t)}G(z)
\zeta(z)\zeta(2z)\G(m,d,z,0)\left(\frac{me}{4r}\right)^{-z}
\Hf(t,z)\frac{dz}{z}.
\end{multline}
 Note that due to our choice of the function $G(z)$ in \eqref{Gdef}, the integrand in \eqref{SIN+SZE asympt2} is holomorphic at both $z=1$ and $z=1/2$. Here $a = \epsilon$, $\G(m,d,z,u)$ is given by \cite[(10.39)]{BF2momWA2026}, and
\begin{multline}\label{Hf def}
\Hf(t,z)=
\frac{\pi}{2}\left(I_{<,0}(t,z)-I_{>,0}(t,z)\right)+\\+
\Bigl(2\frac{d}{du}\log\G(m,d,z,u)\Biggl|_{u=0}-\log(8\pi)+3\gamma-2\log d-2\frac{\zeta'(z)}{\zeta(z)}-
\frac{\partial}{\partial z}\Bigr)\left(I_{<,0}(k,z)+I_{>,0}(k,z)\right).
\end{multline}
In view of \eqref{I>0(t,z)-I>0(-t,z)} and \eqref{I<0(t,z)eq7}, we rewrite \eqref{SIN+SZE asympt2} as
\begin{multline}\label{SIN+SZE asympt3}
\SZE(r;h)+\SIN(r;h)=
\sum_{e|r^2}\frac{e_2}{\sqrt{re}}
\sum_{d|\frac{r}{e_2}}\sum_{m|\frac{r}{de_2}}\frac{\mu(m)}{dm}\frac{2i}{\pi}\int_{0}^{\infty}\frac{th(t)}{\cosh(\pi t)}\\\times
\frac{1}{2\pi i}\int_{(a)}\frac{L_{\infty}(1/2+z,t)}{L_{\infty}(1/2,t)}G(z)
\zeta(z)\zeta(2z)\G(m,d,z,0)\left(\frac{me}{4r}\right)^{-z}
\left(\Hf(t,z)-\Hf(-t,z)\right)\frac{dz}{z}.
\end{multline}
Using \eqref{I>0(t,z)-I>0(-t,z)}, \eqref{I<0(t,z)eq7}, and \cite[(10.74)]{BF2momWA2026}, we  obtain the following analogue of \cite[(10.76),(10.77)]{BF2momWA2026}:
\begin{equation}\label{Hf to Hf2}
\Hf(t,z)-\Hf(-t,z)=
\frac{-i\Gamma^2(z)}{\sqrt{\pi}(2t)^{2z}}e^{\pi t}
\left(1+\sin(\pi z)+\cos(\pi z)\right)\Hf_2(k,z)\left(1+O\left(\frac{1+|z|}{1+t}\right)\right),
\end{equation}
where
\begin{multline}\label{Hf2 def}
\Hf_2(t,z)=
\frac{\pi}{2}\frac{1+\sin(\pi z)-3\cos(\pi z)}{1+\sin(\pi z)+\cos(\pi z)}+
2\frac{d}{du}\log\G(m,d,z,u)\Biggl|_{u=0}-\log(8\pi)+\\+3\gamma-2\log d+2\frac{\zeta'(1-z)}{\zeta(1-z)}-2\log(\pi)+
\psi\left(\frac{z}{2}\right)+\psi\left(\frac{1-z}{2}\right)-2\psi\left(z\right)+2\log(2t).
\end{multline}
Applying \cite[(10.78)]{BF2momWA2026} and the elementary identity
\begin{equation*}
\frac{1+\sin(\pi z)-3\cos(\pi z)}{1+\sin(\pi z)+\cos(\pi z)}-2\frac{\sin(\pi z/2)}{\cos(\pi z/2)}=-1,
\end{equation*}
we obtain
\begin{multline}\label{Hf2 def2}
\Hf_2(t,z)=
\log\frac{t^2}{8\pi^3d^2}+2\frac{d}{du}\log\G(m,d,z,u)\Biggl|_{u=0}+3\gamma+2\frac{\zeta'(1-z)}{\zeta(1-z)}-\frac{\pi}{2}=\\=
\Cc_1(-z)-2\log d+2\frac{d}{du}\log\G(m,d,z,u)\Biggl|_{u=0},
\end{multline}
where $\Cc_1(z)$ is defined in \eqref{C1 def}.
Using \eqref{Vapproximation}, \eqref{Hf to Hf2}, and \eqref{Hf2 def2}, we rewrite \eqref{SIN+SZE asympt3} as follows:
\begin{multline}\label{SIN+SZE asympt4}
\SZE(r;h)+\SIN(r;h)=\frac{1}{\pi^2}
\sum_{e|r^2}\frac{e_2}{\sqrt{re}}
\sum_{d|\frac{r}{e_2}}\sum_{m|\frac{r}{de_2}}\frac{\mu(m)}{dm}\frac{1}{\pi^2}\int_{0}^{\infty}\frac{th(t)e^{\pi t}}{\cosh(\pi t)}\\\times
\frac{1}{2\pi i}\int_{(a)}G(z)
\pi^{-3z/2}\zeta(z)\zeta(2z)\frac{\Gamma\left(1/4+z/2\right)}{\Gamma\left(1/4\right)}\Gamma^2(z)
\left(1+\sin(\pi z)+\cos(\pi z)\right)\\\times
\G(m,d,z,0)\left(\frac{tme}{r}\right)^{-z}\Hf_2(t,z)\left(1+O\left(\frac{|z|}{1+|t|}\right)\right)
\frac{dz}{z}.
\end{multline}
Following the computational steps detailed after \cite[(10.80)]{BF2momWA2026}, we obtain
\begin{multline}\label{zetazeta eq1}
\zeta(z)\zeta(2z)\Gamma^2(z)\Gamma\left(1/4+z/2\right)=\\
=\frac{\pi^{3z}\zeta(1-z)\zeta(1-2z)\Gamma\left(1/4-z/2\right)}{1+\cos(\pi z)-\sin(\pi z)}
\frac{\Gamma\left(3/4-z/2\right)\Gamma\left(1/4+z/2\right)}{\Gamma\left(1/4-z/2\right)\Gamma\left(3/4+z/2\right)}=\\
=\frac{\pi^{3z}\zeta(1-z)\zeta(1-2z)\Gamma\left(1/4-z/2\right)}{1+\cos(\pi z)-\sin(\pi z)}
\frac{\sin\left(\pi(1/4-z/2)\right)}{\sin\left(\pi(1/4+z/2)\right)}=\\=
\frac{\pi^{3z}\zeta(1-z)\zeta(1-2z)\Gamma\left(1/4-z/2\right)}{1+\cos(\pi z)+\sin(\pi z)}.
\end{multline}
This allows us to rewrite \eqref{SIN+SZE asympt4} in the form
\begin{multline}\label{SIN+SZE asympt5}
\SZE(r;h)+\SIN(r;h)=
\sum_{e|r^2}\frac{e_2}{\sqrt{re}}
\sum_{d|\frac{r}{e_2}}\sum_{m|\frac{r}{de_2}}\frac{\mu(m)}{dm}\frac{2}{\pi^2}\int_{0}^{\infty}th(t)
\frac{1}{2\pi i}\int_{(a)}G(z)
\pi^{3z/2}\zeta(1-z)\zeta(1-2z)\\\times\frac{\Gamma\left(1/4-z/2\right)}{\Gamma\left(1/4\right)}
\G(m,d,z,0)\left(\frac{tme}{r}\right)^{-z}\Hf_2(t,z)\left(1+O\left(\frac{1+|z|}{1+|t|}\right)\right)
\frac{dz}{z}.
\end{multline}
Utilizing the notation introduced in \cite[(10.82)]{BF2momWA2026},
\begin{multline}\label{SIN+SZE asympt6}
\SZE(r;h)+\SIN(r;h)=\frac{2}{\pi^2\sqrt{r}}
\int_{0}^{\infty}th(t)
\frac{1}{2\pi i}\int_{(a)}G(z)
\pi^{3z/2}\zeta(1-z)\zeta(1-2z)\\\times\frac{\Gamma\left(1/4-z/2\right)}{\Gamma\left(1/4\right)}
\ES(r,k,z)\left(1+O\left(\frac{1+|z|}{1+|t|}\right)\right)
\frac{dz}{z}.
\end{multline}
Applying \cite[(10.109)]{BF2momWA2026}, we obtain
\begin{multline}\label{SIN+SZE asympt7}
\SZE(r;h)+\SIN(r;h)=\frac{2}{\pi^2\sqrt{r}}
\int_{0}^{\infty}th(t)
\frac{1}{2\pi i}\int_{(a)}G(z)
\pi^{3z/2}\zeta(1-z)\zeta(1-2z)\\\times\frac{\Gamma\left(1/4-z/2\right)}{\Gamma\left(1/4\right)}
\MAP(-z,r)\left(\frac{t}{r}\right)^{-z}\left(1+O\left(\frac{1+|z|}{1+|t|}\right)\right)
\frac{dz}{z},
\end{multline}
where $\MAP(z,r)$ is from \eqref{MT sum e def}. Now it follows from \eqref{2mom MT2 eq4} and \eqref{SIN+SZE asympt7} that
\begin{multline}\label{MT2+SIN+SZE asympt1}
\MT_2(r;h)+\SZE(r;h)+\SIN(r;h)=\frac{2}{\pi^2\sqrt{r}}
\int_{0}^{\infty}th(t)
\frac{1}{2\pi i}\int_{(a)}
\left(H(z)+H(-z)\right)\frac{dz}{z}+\\+O\left(\frac{G(rT)^{\epsilon}}{\sqrt{r}}\right),
\end{multline}
where
\begin{equation}\label{MT2+SZE+SIN Hdef}
H(z)=\pi^{-3z/2}G(z)\frac{\Gamma\left(1/4+z/2\right)}{\Gamma\left(1/4\right)} \zeta(1+2z)\zeta(1+z)\MAP(z,r)\left(\frac{t}{r}\right)^{z}.
\end{equation}
Since the integral in \eqref{MT2+SIN+SZE asympt1} is equal to half of the residue at $z = 0$, we obtain, as in \cite[(11.3)]{BF2momWA2026}, that
\begin{equation}\label{MT2+SIN+SZE asympt2}
\MT_2(r;h)+\SZE(r;h)+\SIN(r;h)=\frac{2}{\pi^2\sqrt{r}}
\int_{0}^{\infty}th(t)\res_{z=0}\frac{H(z)}{z}dt+O\left(\frac{G(rT)^{\epsilon}}{\sqrt{r}}\right).
\end{equation}
Substituting \eqref{MT2+SIN+SZE asympt2} into \eqref{2mom AF1} yields
\begin{equation}\label{2mom AF2}
\M_2(r;h(\cdot))=\frac{2}{\pi^2\sqrt{r}}\int_{0}^{\infty}th(t)\res_{z=0}\frac{H(z)}{z}dt
+O\left(T^{\epsilon}\max(G,T^{2/3})+\frac{T^{2+\epsilon}r^{3/2+\epsilon}}{G^{5/2}}\right).
\end{equation}
It remains to evaluate the residue in \eqref{2mom AF2}. This evaluation can be carried out in an identical manner to the analysis at the end of \cite[Section 11]{BF2momWA2026}, since the only differences are as follows: compared to \cite[(11.2)]{BF2momWA2026}, our expression \eqref{MT2+SZE+SIN Hdef} contains an additional factor of $G(z)$, the argument of the Gamma function in \eqref{MT2+SZE+SIN Hdef} involves $1/4$ instead of $3/4$, and the constant in \eqref{C1 def} contains $-\pi/2$ instead of $\pi/2$. These modifications do not alter the structural form of the main term; they merely shift the coefficients of the underlying polynomials. Consequently, this establishes \eqref{2momAF0}.

The proof of Theorem \ref{thm:nonvan} proceeds exactly as the proof of \cite[Theorem 1.1]{BF2mom} and relies on Theorem \ref{main thm 3mom} (cf. \cite[Section 6]{BF2mom}). The admissibility condition on $\beta$ arises from the requirement that the error term in \eqref{2momAF0} for $r = T^{2\Delta}$ must be strictly smaller than the main term; that is,

$$\frac{T^{2+\epsilon}r^{3/2}}{G^{5/2}}=T^{2+3\Delta-5\beta/2+\epsilon}<\frac{TG}{\sqrt{r}}=T^{1+\beta-\Delta}.$$
This yields $4\Delta<7\beta/2-1$.

%%%%%%%%%%%%%%%%%%%%%%%%%%%%%%%%%%%%%%%%%%%%%%%%%%%%%%%%%%%%%%%%%%%%%%%%%%%%%%%%%%%%%%%%%%%
\section{Proof of Theorem \ref{main thm 3mom}}\label{sec:3mom}
To prove Theorem \ref{main thm 3mom}, we represent one of the three $L$-functions using the approximate functional equation \eqref{approx.func.eq.}, which yields
\begin{equation}\label{M3 to M2}
\M_3(1,1/2)=2\sum_{r\ll T^{1+\epsilon}}\frac{1}{r^{1/2}}\M_2(r;h(\cdot)V(r,\cdot))+O(T^{-A}),
\end{equation}
where $\M_2(m;\cdot)$  is given by \eqref{2mom def}.  Next, we substitute \eqref{2mom AF2} into \eqref{M3 to M2}. 
Summing the error terms, we find that the total contribution is bounded by
\begin{equation}\label{M3 to M2 error}
\sum_{r\ll T^{1+\epsilon}}\frac{1}{r^{1/2}}\left(T^{\epsilon}\max(G,T^{2/3})+\frac{T^{2+\epsilon}r^{3/2+\epsilon}}{G^{5/2}}\right)\ll
\frac{T^{4+\epsilon}}{G^{5/2}}+T^{1/2+\epsilon}\max(G,T^{2/3})\ll\frac{T^{4+\epsilon}}{G^{5/2}}
\end{equation}
for $G\ll T^{1-\epsilon}$. Therefore,
\begin{equation}\label{M3 to M2 eq2}
\M_3(1,1/2)=\MT_3(1;h)+O\left(\frac{T^{4+\epsilon}}{G^{5/2}}\right),
\end{equation}
where
\begin{equation}\label{MT3 def}
\MT_3(1;h)=\frac{4}{\pi^2}\sum_{r=1}^{\infty}\frac{1}{r}
\int_{0}^{\infty}th(t)V(r,t)\res_{z=0}\frac{H(z)}{z}dt.
\end{equation}
We do not compute the precise form of the polynomial in \eqref{3momAF0}, but merely show that its degree is precisely six. To evaluate the main term, we express the residue in \eqref{MT3 def} via Cauchy's integral formula and utilize \eqref{MT2+SZE+SIN Hdef}, \eqref{MT sum e def}, and \eqref{C1 def}, which yields
\begin{multline}\label{MT3 eq1}
\MT_3(1;h)=\frac{4}{\pi^2}\int_{0}^{\infty}th(t)
\sum_{e_1=1}^{\infty}\sum_{e_2=1}^{\infty}\frac{|\mu(e_1)|}{e_1^{3/2}e_2}
\sum_{r=1}^{\infty}\frac{V(re_1e_2,t)}{r}\frac{1}{2\pi i}\int_{C_{\epsilon}}\pi^{-3z/2}G(z)\\\times
\frac{\Gamma\left(1/4+z/2\right)}{\Gamma\left(1/4\right)}\left(\frac{t}{e_1e_2}\right)^{z}
\Biggl( 2\zeta(1+2z)\zeta'(1+z)+\\+\zeta(1+2z)\zeta(1+z)\left(2\log(t)-\log(8\pi^3)+3\gamma-\frac{\pi}{2}-2\log(re_1)\right)\Biggr)\frac{dz}{z}dt,
\end{multline}
where $C_{\epsilon}$ is a circle of radius $\epsilon$ with the center at $z=0$. Let
\begin{equation}\label{GV def}
G_V(z):=\pi^{-3z/2}G(z)\frac{\Gamma\left(1/4+z/2\right)}{\Gamma\left(1/4\right)},\quad
\Cc_2(t):=\log\frac{t^2}{8\pi^3}+3\gamma-\frac{\pi}{2}.
\end{equation}
Substituting \eqref{Vapproximation} into \eqref{MT3 eq1}, we obtain, up to an error term of $O(G^{1+\epsilon})$ that is strictly smaller than the one in \eqref{M3 to M2 eq2},
\begin{multline}\label{MT3 eq2}
\MT_3(1;h)=\frac{4}{\pi^2}\int_{0}^{\infty}th(t)
\frac{1}{2\pi i}\int_{(a)}\frac{t^v}{v}G_V(v)
\zeta(1+2v)\frac{1}{2\pi i}\int_{C_{\epsilon}}\frac{t^z}{z}G_V(z)\\\times
\Biggl(\left(\zeta(1+2z)\zeta(1+z)\Cc_2(t)+ 2\zeta(1+2z)\zeta'(1+z)\right)
\sum_{e_1,e_2,r=1}^{\infty}\frac{|\mu(e_1)|}{e_1^{3/2+v+z}e_2^{1+v+z}r^{1+v}}+\\
+2\zeta(1+2z)\zeta(1+z)\frac{d}{du}\sum_{e_1,e_2,r=1}^{\infty}\frac{|\mu(e_1)|}{e_1^{3/2+v+z+u}e_2^{1+v+z}r^{1+v+u}}\Biggl|_{u=0}
\Biggr)dvdzdt.
\end{multline}
Note that
\begin{equation}\label{e1e2r sum1}
\sum_{e_1,e_2,r=1}^{\infty}\frac{|\mu(e_1)|}{e_1^{3/2+v+z}e_2^{1+v+z}r^{1+v}}=\zeta(1+v+z)\zeta(1+v)\frac{\zeta(3/2+v+z)}{\zeta(3+2v+2z)},
\end{equation}
\begin{equation}\label{e1e2r sum2}
\frac{d}{du}\sum_{e_1,e_2,r=1}^{\infty}\frac{|\mu(e_1)|}{e_1^{3/2+v+z+u}e_2^{1+v+z}r^{1+v+u}}\Biggl|_{u=0}=
\zeta(1+v+z)\frac{d}{du}\zeta(1+v+u)\frac{\zeta(3/2+v+z+u)}{\zeta(3+2v+2z+2u)}\Biggl|_{u=0}.
\end{equation}
Substituting \eqref{e1e2r sum1}, \eqref{e1e2r sum2} into \eqref{MT3 eq2}, and evaluating the $u$-derivative, we find that
\begin{multline}\label{MT3 eq3}
\MT_3(1;h)=\frac{4}{\pi^2}\int_{0}^{\infty}th(t)
\frac{1}{(2\pi i)^2}\int_{C_{\epsilon}}\frac{t^z}{z}G_V(z)\int_{(a)}\frac{t^v}{v}G_V(v)
\zeta(1+2v)\\\times
\Biggl(
2\zeta(1+2z)\zeta'(1+z)\zeta(1+v+z)\zeta(1+v)\frac{\zeta(3/2+v+z)}{\zeta(3+2v+2z)}+\\+
\Cc_2(t)\zeta(1+2z)\zeta(1+z)\zeta(1+v+z)\zeta(1+v)\frac{\zeta(3/2+v+z)}{\zeta(3+2v+2z)}+\\
+2\zeta(1+2z)\zeta(1+z)\zeta(1+v+z)\Bigl(
\zeta(1+v)\frac{\zeta'(3/2+v+z)}{\zeta(3+2v+2z)}-\\-2\zeta(1+v)\zeta'(3+2v+2z)\frac{\zeta(3/2+v+z)}{\zeta^2(3+2v+2z)}+
\zeta'(1+v)\frac{\zeta(3/2+v+z)}{\zeta(3+2v+2z)}
\Bigr)
\Biggr)dvdzdt.
\end{multline}
In the $v$-integral, we shift the line of integration to $\Re(v) = -1/2+\epsilon$, crossing the poles at $v = 0$ and $v = -z$. Bounding the new integral trivially, we obtain a contribution of $O(T^{1/2+\epsilon}G)$, which is dominated by the error term in \eqref{M3 to M2 eq2}. The residue at $v = -z$ then gives rise to the following expression:
\begin{multline}\label{MT3 res v=-z}
-\frac{4}{\pi^2}\int_{0}^{\infty}th(t)
\frac{1}{2\pi i}\int_{C_{\epsilon}}\frac{1}{z^2}G^2_V(z)\zeta(1-2z)
\Biggl(2\zeta(1+2z)\zeta'(1+z)\zeta(1-z)\frac{\zeta(3/2)}{\zeta(3)}+\\+
\Cc_2(t)\zeta(1+2z)\zeta(1+z)\zeta(1-z)\frac{\zeta(3/2)}{\zeta(3)}+
2\zeta(1+2z)\zeta(1+z)\Bigl(
\zeta(1-z)\frac{\zeta'(3/2)}{\zeta(3)}-\\-2\zeta(1-z)\zeta'(3)\frac{\zeta(3/2)}{\zeta^2(3)}+
\zeta'(1-z)\frac{\zeta(3/2)}{\zeta(3)}
\Bigr)
\Biggr)dzdt=\frac{4}{\pi^2}\int_{0}^{\infty}th(t)P_1(\log t)dt,
\end{multline}
where the final assertion follows from the fact that within the $z$-integrand in \eqref{MT3 res v=-z}, there are no powers of $t^z$, and the entire $t$-dependence is contained within $\Cc_2(t)$, which is a polynomial in $\log t$ of degree one (cf. \eqref{GV def}). It remains to evaluate the residue at $v = 0$ of the integrand in \eqref{MT3 eq3}. Let
\begin{equation}\label{F1 def}
F_1(z,v)=t^vG_V(v)\zeta(1+v+z)\frac{\zeta(3/2+v+z)}{\zeta(3+2v+2z)},
\end{equation}
\begin{equation}\label{F2 def}
F_2(z,v)=t^vG_V(v)\zeta(1+v+z)\left(
\frac{\zeta'(3/2+v+z)}{\zeta(3+2v+2z)}-2\zeta'(3+2v+2z)\frac{\zeta(3/2+v+z)}{\zeta^2(3+2v+2z)}
\right).
\end{equation}
We have
\begin{equation}\label{zetazeta/v}
\frac{\zeta(1+v)\zeta(1+2v)}{v}=\frac{1}{2v^3}+\frac{3\gamma}{2v^2}+\frac{\gamma^2-5\gamma_1/2}{v}+O(1),
\end{equation}
\begin{equation}\label{zetazeta'/v}
\frac{\zeta'(1+v)\zeta(1+2v)}{v}=
-\frac{1}{2v^4}-\frac{\gamma}{v^3}+\frac{3\gamma_1}{2v^2}-\frac{\gamma\gamma_1+3\gamma_2/2}{v}+O(1),
\end{equation}
Using \eqref{F1 def} and \eqref{zetazeta/v}, we obtain
\begin{multline}\label{Residue1 v=0}
\res_{v=0}\frac{\zeta(1+2v)\zeta(1+v)}{v}G_V(v)t^v\zeta(1+v+z)\frac{\zeta(3/2+v+z)}{\zeta(3+2v+2z)}=\\=
\res_{v=0}\frac{\zeta(1+2v)\zeta(1+v)}{v}F_1(z,v)=
\frac{1}{4}F_1''(z,0)+\frac{3\gamma}{2}F_1'(z,0)+(\gamma^2-5\gamma_1/2)F_1(z,0).
\end{multline}
Applying \eqref{F2 def} and \eqref{zetazeta/v}, we prove that
\begin{multline}\label{Residue3 v=0}
\res_{v=0}\frac{\zeta(1+v)\zeta(1+2v)}{v}t^vG_V(v)\zeta(1+v+z)\\\times
\left(\frac{\zeta'(3/2+v+z)}{\zeta(3+2v+2z)}-2\zeta'(3+2v+2z)\frac{\zeta(3/2+v+z)}{\zeta^2(3+2v+2z)}\right)=
\res_{v=0}\frac{\zeta(1+v)\zeta(1+2v)}{v}F_2(z,v)=\\=
\frac{1}{4}F_2''(z,0)+\frac{3\gamma}{2}F_2'(z,0)+(\gamma^2-5\gamma_1/2)F_2(z,0).
\end{multline}
Combining \eqref{F1 def} and \eqref{zetazeta'/v} yields
\begin{multline}\label{Residue4 v=0}
\res_{v=0}\frac{\zeta'(1+v)\zeta(1+2v)}{v}t^vG_V(v)\zeta(1+v+z)\frac{\zeta(3/2+v+z)}{\zeta(3+2v+2z)}=\\=
\res_{v=0}\frac{\zeta'(1+v)\zeta(1+2v)}{v}F_1(z,v)=
\frac{-1}{12}F_1'''(z,0)-\frac{\gamma}{2}F_1''(z,0)+\\+\frac{3\gamma_1}{2}F_1'(z,0)-(\gamma\gamma_1+3\gamma_2/2)F_1(z,0).
\end{multline}
It follows from \eqref{MT3 eq3}, \eqref{Residue1 v=0}, \eqref{Residue3 v=0}, and \eqref{Residue4 v=0} that the contribution of the residue at $v=0$ to the main term is
\begin{multline}\label{MT3 res v=0 eq1}
\frac{4}{\pi^2}\int_{0}^{\infty}th(t)
\frac{1}{2\pi i}\int_{C_{\epsilon}}\frac{t^z}{z}G_V(z)
\Biggl(
\left(2\zeta(1+2z)\zeta'(1+z)+\Cc_2(t)\zeta(1+2z)\zeta(1+z)\right)\\\times
\left(\frac{1}{4}F_1''(z,0)+\frac{3\gamma}{2}F_1'(z,0)+(\gamma^2-5\gamma_1/2)F_1(z,0)\right)+
2\zeta(1+2z)\zeta(1+z)\\\times\Bigl(
\frac{1}{4}F_2''(z,0)+\frac{3\gamma}{2}F_2'(z,0)+(\gamma^2-5\gamma_1/2)F_2(z,0)-
\frac{1}{12}F_1'''(z,0)-\frac{\gamma}{2}F_1''(z,0)+\\+
\frac{3\gamma_1}{2}F_1'(z,0)-(\gamma\gamma_1+3\gamma_2/2)F_1(z,0)\Bigr)
\Biggr)dzdt.
\end{multline}
It remains to show that the residue at $z = 0$ of the integrand in \eqref{MT3 res v=0 eq1} is equal to $P_6(\log t)$. To simplify the presentation, we consider only the summands in \eqref{MT3 res v=0 eq1} that give rise to the leading power of $\log t$. In doing so (cf. \eqref{F1 def}), we can replace
\begin{equation}\label{F1 simpl1.1}
\frac{1}{4}F_1''(z,0)+\frac{3\gamma}{2}F_1'(z,0)+(\gamma^2-5\gamma_1/2)F_1(z,0)
\end{equation}
by
\begin{equation}\label{F1 simpl1.2}
\left(\zeta''(1+z)+\zeta(1+z)\log^2t+2\zeta'(1+z)\log t\right)\frac{\zeta(3/2+z)}{4\zeta(3+2z)}.
\end{equation}
Based on \eqref{F2 def}, we can replace
\begin{equation}\label{F2 simpl1}
\frac{1}{4}F_2''(z,0)+\frac{3\gamma}{2}F_2'(z,0)+(\gamma^2-5\gamma_1/2)F_2(z,0)
\end{equation}
by
\begin{equation}\label{F2 simpl2}
\frac{1}{4}\left(\zeta''(1+z)+\zeta(1+z)\log^2t+2\zeta'(1+z)\log t\right)
\left(\frac{\zeta'(3/2+z)}{\zeta(3+2z)}-2\zeta'(3+2z)\frac{\zeta(3/2+z)}{\zeta^2(3+2z)}\right).
\end{equation}
Finally, using \eqref{F1 def}, we can replace
\begin{equation}\label{F1 simpl2.1}
-\frac{1}{12}F_1'''(z,0)-\frac{\gamma}{2}F_1''(z,0)+\\+
\frac{3\gamma_1}{2}F_1'(z,0)-(\gamma\gamma_1+3\gamma_2/2)F_1(z,0)
\end{equation}
by
\begin{equation}\label{F1 simpl2.2}
-\left(\zeta'''(1+z)+3\zeta''(1+z)\log t+3\zeta'(1+z)\log^2t+\zeta(1+z)\log^3t\right)\frac{\zeta(3/2+z)}{12\zeta(3+2z)}.
\end{equation}
As a result, transforming \eqref{MT3 res v=0 eq1} in this manner and taking into account \eqref{MT3 res v=-z}, we obtain (cf. \eqref{MT3 eq3})
\begin{multline}\label{MT3 res v=0 eq2}
\MT_3(1;h)=\frac{4}{\pi^2}\int_{0}^{\infty}th(t)
\frac{1}{2\pi i}\int_{C_{\epsilon}}\frac{t^z}{z}G_V(z)
\Biggl(
\left(2\zeta(1+2z)\zeta'(1+z)+\Cc_2(t)\zeta(1+2z)\zeta(1+z)\right)\\\times
\left(\zeta''(1+z)+\zeta(1+z)\log^2t+2\zeta'(1+z)\log t\right)\frac{\zeta(3/2+z)}{4\zeta(3+2z)}+
2\zeta(1+2z)\zeta(1+z)\\\times\Bigl(
\frac{1}{4}\left(\zeta''(1+z)+\zeta(1+z)\log^2t+2\zeta'(1+z)\log t\right)
\left(\frac{\zeta'(3/2+z)}{\zeta(3+2z)}-2\zeta'(3+2z)\frac{\zeta(3/2+z)}{\zeta^2(3+2z)}\right)-\\
-\left(\zeta'''(1+z)+3\zeta''(1+z)\log t+3\zeta'(1+z)\log^2t+\zeta(1+z)\log^3t\right)\frac{\zeta(3/2+z)}{12\zeta(3+2z)}\Bigr)
\Biggr)dzdt+\\+
\frac{4}{\pi^2}\int_{0}^{\infty}th(t)P_5(\log t)dt+O(T^{1/2+\epsilon}G).
\end{multline}
It remains to compute the residue at the point $z = 0$. Since we are seeking only the leading power of $\log t$, we can replace $\Cc_2(t)$ by $2 \log t$ (cf. \eqref{GV def}). Furthermore, this term dominates the expression containing the brackets, i.e.
$$\left(\frac{\zeta'(3/2+z)}{\zeta(3+2z)}-2\zeta'(3+2z)\frac{\zeta(3/2+z)}{\zeta^2(3+2z)}\right).$$
Accordingly,
\begin{multline}\label{MT3 res v=0 eq3}
\MT_3(1;h)=\frac{2}{\pi^2}\int_{0}^{\infty}th(t)
\frac{1}{2\pi i}\int_{C_{\epsilon}}\frac{t^z}{z}G_V(z)
\Biggl(
\left(\zeta(1+2z)\zeta'(1+z)+\zeta(1+2z)\zeta(1+z)\log(t)\right)\\\times
\left(\zeta''(1+z)+\zeta(1+z)\log^2t+2\zeta'(1+z)\log t\right)\frac{\zeta(3/2+z)}{\zeta(3+2z)}-
\zeta(1+2z)\zeta(1+z)\\\times
\left(\zeta'''(1+z)+3\zeta''(1+z)\log t+3\zeta'(1+z)\log^2t+\zeta(1+z)\log^3t\right)\frac{\zeta(3/2+z)}{3\zeta(3+2z)}
\Biggr)dzdt+\\+
\frac{4}{\pi^2}\int_{0}^{\infty}th(t)P_5(\log t)dt+O(T^{1/2+\epsilon}G).
\end{multline}
A straightforward calculation leads to
\begin{multline}\label{MT3 res z=0 eq1}
\frac{1}{z}
\left(\zeta(1+2z)\zeta'(1+z)+\zeta(1+2z)\zeta(1+z)\log(t)\right)\\\times
\left(\zeta''(1+z)+\zeta(1+z)\log^2t+2\zeta'(1+z)\log t\right)
=\sum_{j=1}^{7}\frac{a_{-j}}{z^j}+O(1),
\end{multline}
\begin{multline}\label{MT3 res z=0 eq2}
\frac{-1}{3z}\zeta(1+2z)\zeta(1+z)
\left(\zeta'''(1+z)+3\zeta''(1+z)\log t+3\zeta'(1+z)\log^2t+\zeta(1+z)\log^3t\right)
=\\=\sum_{j=1}^{7}\frac{b_{-j}}{z^j}+O(1),
\end{multline}
where
\begin{equation}\label{MT3 res z=0 eq1a}
a_{-7}=-1,\,a_{-6}=2\log t-2\gamma,\, a_{-5}=-\frac{3}{2}\log^2t+P_1(\log t),\,
a_{-4}=\frac{\log^3t}{2}+P_2(\log t),
\end{equation}
\begin{equation*}
a_{-3},a_{-2},a_{-1}\ll\log^3 t,
\end{equation*}
\begin{equation}\label{MT3 res z=0 eq2b}
b_{-7}=1,\,b_{-6}=-\log t+3\gamma,\, b_{-5}=\frac{\log^2t}{2}+P_1(\log t),\,
b_{-4}=\frac{-\log^3t}{6}+P_2(\log t),
\end{equation}
\begin{equation*}
b_{-3},b_{-2},b_{-1}\ll\log^3 t.
\end{equation*}
Therefore, \eqref{MT3 res v=0 eq3} is equal to
\begin{multline}\label{MT3 res v=0 eq4}
\MT_3(1;h)=\frac{2}{\pi^2}\int_{0}^{\infty}th(t)
\res_{z=0}\Biggl(t^zG_V(z)\frac{\zeta(3/2+z)}{\zeta(3+2z)}\\\times
\Biggl(
\frac{\log t}{z^6}-\frac{\log^2 t}{z^5}+\frac{\log^3t}{3z^4}+\sum_{j=1}^{3}\frac{p_{3,j}(\log t)}{z^j}+O(1)
\Biggr)
\Biggr)dt+\\+
\int_{0}^{\infty}th(t)P_5(\log t)dt+O(T^{1/2+\epsilon}G).
\end{multline}
Expanding $t^z = 1 + \sum_{j=1}^5 \frac{\log^j t}{j!} z^j + O(z^6)$, and noting that $G_V(0) = 1$, we obtain
\begin{multline}\label{MT3 res v=0 eq5}
\MT_3(1;h)=\frac{2}{\pi^2}\frac{\zeta(3/2)}{\zeta(3)}\int_{0}^{\infty}th(t)
\left(\frac{1}{5!}-\frac{1}{4!}+\frac{1}{3\cdot3!}\right)\log^6tdt+
\int_{0}^{\infty}th(t)P_5(\log t)dt=\\=
\frac{2}{45\pi^2}\frac{\zeta(3/2)}{\zeta(3)}\int_{0}^{\infty}th(t)\left(\log^6t+P_5(\log t)\right)dt+O(T^{1/2+\epsilon}G).
\end{multline}
Substituting \eqref{MT3 res v=0 eq5} into \eqref{M3 to M2 eq2} yields \eqref{3momAF0}, which completes the proof of Theorem \ref{main thm 3mom}.

%%%%%%%%%%%%%%%%%%%%%%%%%%%%%%%%%%%%%%%%%%%%%%%%%%%%%%%%%%%%%%%%%%%%%%%%%%%%%%%%%%%%%%%%%%%

%%%%%%%%%%%%%%%%%%%%%%%%%%%%%%%%%%%%%%%%%%%%%%%%%%%%%%%%%%%%%%%%%%%%%%%%%%%%%%%%%%%%%%%%%%%

%%%%%%%%%%%%%%%%%%%%%%%%%%%%%%%%%%%%%%%%%%%%%%%%%%%%%%%%%%%%%%%%%%%%%%%%%%%%%%%%%%%%%%%%%%%

%%%%%%%%%%%%%%%%%%%%%%%%%%%%%%%%%%%%%%%%%%%%%%%%%%%%%%%%%%%%%%%%%%%%%%%%%%%%%%%%%%%%%%%%%%%

\end{document}